\documentclass{IEEEtran}

\usepackage{cite}
\usepackage{amsmath,amssymb,amsfonts,amsthm,amsthm}
\usepackage{algorithmic}
\usepackage{graphicx}
\usepackage{textcomp}
\def\BibTeX{{\rm B\kern-.05em{\sc i\kern-.025em b}\kern-.08em
    T\kern-.1667em\lower.7ex\hbox{E}\kern-.125emX}}

\usepackage{booktabs}

\usepackage[dvipsnames]{xcolor}
\definecolor{myred}{HTML}{c20014}
\definecolor{mygreen}{HTML}{008000}

\usepackage{amsmath}
\usepackage{amsfonts}
\usepackage{caption}
\usepackage{array,multirow,makecell}
\usepackage{amsthm}
\usepackage{amssymb}
\usepackage{textcomp}

\usepackage{verbatim}
\usepackage{dsfont}
\usepackage[inline]{enumitem}
\setlist{noitemsep}
\usepackage{cite}

\usepackage{tikz-cd}

\usepackage{url}
\usepackage[colorlinks=true,linkcolor=myred,citecolor=mygreen]{hyperref}
\hypersetup{linktocpage}
\usepackage[capitalise]{cleveref}

\usepackage{tabularx}

\crefname{equation}{}{}
\Crefname{equation}{Equation}{Equations}

\newtheorem{theo}{Theorem}[section]
\newtheorem{prop}[theo]{Proposition}
\newtheorem{coro}[theo]{Corollary}
\newtheorem{lemma}[theo]{Lemma}

\theoremstyle{definition}
\newtheorem{defi}[theo]{Definition}

\theoremstyle{remark}
\newtheorem{rem}[theo]{Remark}
\newtheorem{ex}[theo]{Example}

\newtheorem{as}[theo]{Assumption}

\crefname{theo}{Theorem}{Theorems}
\crefname{claim}{Claim}{Claims}
\crefname{defi}{Definition}{Definitions}
\crefname{as}{Assumption}{Assumptions}
\crefname{enumi}{}{}
\Crefname{enumi}{Item}{Items}
\crefname{ex}{Example}{Examples}

\newcommand{\dt}[1]{\frac{\mathrm{d}#1}{\mathrm{d}t}}
\newcommand{\dl}[2]{\frac{\partial#1}{\partial#2}}

\newcommand{\A}{\mathcal{A}}

\newcommand{\B}{\mathcal{B}}
\renewcommand{\d}{\mathrm{d}}

\newcommand{\K}{\mathcal{K}}

\renewcommand{\S}{\mathcal{S}}

\newcommand{\R}{\mathbb{R}}

\newcommand{\N}{\mathbb{N}}
\newcommand{\C}{\mathcal{C}}

\DeclareMathOperator{\dist}{dist}

\DeclareMathOperator{\id}{id}

\newcommand{\dom}{\operatorname{dom}}
\newcommand{\ran}{\operatorname{ran}}

\newcommand{\sg}[1]{\{#1_t\}_{t \geq 0}}

\newcommand{\para}[1]{\medskip \noindent {\bfseries #1.}}

\newcommand{\bbm}[1]{\left[\begin{matrix} #1 \end{matrix}\right]}

\newcommand{\vnabla}{\vec{\nabla}}

\renewcommand{\leq}{\leqslant}
\renewcommand{\geq}{\geqslant}
\renewcommand{\epsilon}{\varepsilon}
\newcommand{\eps}{\varepsilon}

\begin{document}

\title{Projected Integral Control of Impedance Passive Nonlinear Systems}
\author{Nicolas Vanspranghe$^\dagger$, Pietro Lorenzetti$^\ddagger$, Lassi Paunonen$^\dagger$, George Weiss$^\mathsection$
\thanks{
This work was supported by the Research Council of Finland grant 349002. 
}
\thanks{$^\dagger$Mathematics Research Centre, Tampere 
University, PO Box 692, 33101 Tampere, Finland (e-mails:
 lassi.paunonen@tuni.fi, nicolas.vanspranghe@tuni.fi). $^\ddagger$ Universit\'e de Lorraine, CNRS, CRAN, F-54000 Nancy, France (e-mail: pietro.lorenzetti@univ-lorraine.fr). $^\mathsection$School of EE, 
Tel Aviv University, Ramat Aviv, Israel (e-mail: gweiss@tauex.tau.ac.il).
}
}

\maketitle

\begin{abstract}
We propose an abstract framework for solving the constrained set-point tracking problem for impedance passive infinite-dimensional nonlinear systems. The class of systems considered is
governed by  monotone differential inclusions and
allows us to exploit the theory of contraction semigroups.
To account for possible operational constraints, e.g., bounds on the input,
we replace a classical integral controller with a projected integral controller.
This guarantees that the integrator state  remains in a  given closed convex set, where said constraints are satisfied. We showcase our results through three case studies.
\end{abstract}


\section{Introduction}
\label{sec:intro}

\emph{Set-point tracking} is a relevant problem in many engineering applications, where the goal is to control a given system in such a way that its output asymptotically tracks a constant reference 
and closed-loop stability is guaranteed. This problem is traditionally solved by employing an integral controller coupled with a closed-loop stabilizer; see, e.g., \cite{Davison76,FraWon76} for linear systems and \cite{Isi97,IsiByr90} for nonlinear systems. 
Depending on the system's properties, different control design can be used to ensure  closed-loop stability. For systems  that are already stable and exhibit a monotone input-output behavior at steady state, 
\emph{low-gain integral control} \cite{Mor85,DesLin85,FliLog03,Sim20} is a simple yet effective way to achieve (at least) local stability. Alternatively,
state feedback stabilizers based on \emph{forwarding} techniques  may be used to guarantee local or global convergence
\cite{MazPra96,AstPra16,GiaAst21,BalMar23,VanBri23,VanBri23conf}.

A notable feature of many physical systems is
\emph{passivity}\cite{Wil72,VdS00,OrtVdS01}. 
A system is {passive} if the rate of increase of energy stored in the system cannot exceed the power supplied to it. Up to suitable choices of input and output, many physical systems 
are passive; see, e.g., \cite{OrtLor98,BroLoz20}. Passive systems enjoy 
useful properties. For instance, their stability may be conveniently formulated in terms of \emph{detectability} \cite{SepJan12}, and power-preserving interconnections of passive systems remain passive \cite{Wil72}. 

\emph{Passivity-based control} builds upon such properties. In the context of set-point tracking, the key observation is that an integral controller is a passive system itself, meaning that a power-preserving interconnection with the passive plant to be controlled leads to a passive closed-loop system. This fact has been used to good advantage in the control literature; see, e.g.,
\cite{JayOrt07,JayWei09,OrtvdS18,OrtRom21}, where proportional-integral control laws for passive nonlinear (finite-dimensional) systems are studied. In contrast with low-gain integral control techniques, passivity-based integral control does not impose an upper bound on the controller gain. 
Moreover, ensuring closed-loop \emph{global} asymptotic stability is often easier when leveraging passivity.
 
The notion of passivity extends naturally to infinite-dimensional systems. In many systems governed by partial differential equations,
the energy stored in the system is given by a quadratic function of the state.
This particular choice leads to what are known as \textit{impedance passive} {systems}; see, e.g., \cite{SinWei22,SinWei25,Sta02,TucWei14,ZhaWei17}. For infinite-dimensional linear impedance passive systems, passivity-based control has been used in \cite{RebWei03,CurWei06,Pau19} in order to achieve set-point tracking and, more generally, output regulation.
Although the theory is well-understood for linear systems, there is a lack of design and analysis methods for passive infinite-dimensional nonlinear systems.

We address this gap by proposing a general framework for solving the set-point tracking problem for impedance passive infinite-dimensional nonlinear systems. Our objective is to narrow down the least amount of assumptions needed to ensure closed-loop stability and set-point tracking. 
We take advantage of the theory of monotone differential inclusions \cite{Bre73book,Bar76book} and we show how to encompass many relevant classes of systems.
This allows us to consider features beyond the scope of previous theory for infinite-dimensional systems on set-point tracking, such as multivalued dynamics and nonlinear boundary input. 
In the spirit of passivity, we employ new stability arguments based on a notion of \emph{internal energy dissipation} and establish both open-loop and closed-loop sufficient conditions for global asymptotic stability.

In order to guarantee that possible operating constraints are satisfied (e.g., actuator saturation or safety constraints), we leverage the theory of projected dynamical systems \cite{BroDan06,BroTan20,DelCor24,Nag12,SinFue24} and replace the standard output integrator with a \emph{projected integral controller} \cite{LorWei23,LorWei22,LorPau2023}. This forces the integrator state to remain in 
a chosen convex set, guaranteeing that the aforementioned  constraints are satisfied.
For instance, choosing the constraint set in accordance with  actuator saturation limits 
constitutes a built-in \emph{anti-windup} mechanism; see, e.g., \cite{TarTur09survey} for an overview of anti-windup designs in finite dimension. 
Furthermore,
projected integral control allows us to  regulate systems that are well-behaved   only for certain sets of inputs.

We mention that, when particularized to finite-dimensional nonlinear systems, our theory relates to, e.g., \cite{Jay05,JayOrt07,TanBro18}. Compared to \cite{Jay05,JayOrt07}, we allow for multivalued dynamics and convex input constraints, but we consider quadratic storage functions only. Moreover, the class of systems considered here is neither a subset nor a superset of the one from \cite{TanBro18}, where  
a more general {output regulation} problem is studied. 

The paper is organized as follows. In \cref{sec:2} we describe the class of systems that we are interested in controlling.
The set-point tracking problem is presented in \cref{sec:PI},
where the projected integrator is introduced. The closed-loop stability analysis is carried out in \cref{sec:4}. Finally, {three} case studies
are presented in \cref{sec:5}.


\para{Notation} The closed half-line $[0, +\infty)$ is denoted by $\R^+$. In this paper, all vector spaces are real. Let $E$ be a Banach space. The norm of $E$ is denoted by $\|\cdot\|_E$. If $E$ is also a Hilbert space, its scalar product is written $\langle \cdot, \cdot \rangle_E$. We denote by $\mathcal{C}(I,E)$ the space of continuous functions from $I\subset\R$ to E. Given a real $T > 0$, $L^1(0, T; E)$ denotes the space of (equivalence classes of) Bochner-measurable functions $u :(0, T) \to E$ such that $\|u\|_E$ is integrable. If $E$ is a Hilbert space (or, more generally, a reflexive Banach space), we say that a function $u : [0, T] \to E$ is absolutely continuous on $[0, T]$ if there exists $v \in L^1(0, T; E)$ such that $u(t) = \int_0^t v(s) \, \d s$. Then, $u$ is continuous and differentiable almost everywhere (a.e.) in $(0, T)$, with $\dot{u} = v$ a.e.  We will say that a function $u : \R^+ \to E$ is absolutely continuous on $\R^+$ if it is absolutely continuous on every finite interval $[0, T]$, $T > 0$. For more details on vector-valued functions and integration, the reader is referred to \cite{Ama95book}. Now, let $E$ be a metric space. The closure and the interior of a subset $S \subset E$ are denoted by $\overline{S}$ and $S^\circ$ respectively. Given another metric space $F$, a mapping $f : E \to F$ is said to be compact if it maps bounded subsets of $E$ into relatively compact subsets of $F$, i.e., subsets whose closures are compact. 
Finally, let $E$ and $F$ be vector spaces. A subset $A$ of the Cartesian product $E \times F$ can be viewed as a multivalued map $A : E \rightrightarrows F$. For $x \in E$, we let $A(x) \triangleq \{ y \in F : (x, y) \in A\}$. The domain and range of $A$ are the sets $\dom(A) \triangleq \{ x \in E : A(x) \not = \emptyset\}$ and $\ran(A) \triangleq \cup_{x \in \dom(A)} A(x)$. Rules for basic algebraic manipulations on multivalued maps can be found in, e.g., \cite{Bar76book,RocWet09}. In the case that $E$ is a Hilbert space, the definition of a maximal dissipative operator $A : E \rightrightarrows E$ (i.e., maximal monotone $-A$) can be found in Appendix~\ref{ap:sg}, together with other useful results from the literature that we will use.

\section{The class of systems under consideration}\label{sec:2}
Let $C$ be a nonempty closed convex subset of a Hilbert space $X$. We are interested in contractive nonlinear dynamics  governed by a differential inclusion of the form
\begin{equation}
\label{eq:open-loop-DE}
\dot{x} \in A(x, u),
\end{equation}
where the $X$-valued state $x$ evolves in $C$, the input $u$ takes values in another Hilbert space $U$ 
and $A : X \times U \rightrightarrows X$ is a possibly multivalued map.
In this section, we precisely define the class of systems that we wish to control by stating and discussing a number of assumptions on the abstract model \cref{eq:open-loop-DE}. \Cref{table:as} summarizes those assumptions.
In the following, we use the convention that every assumption is (and will remain) in force
immediately after {its introduction}.

\begin{table*}
\small
\centering
\begin{tabularx}{500pt}{|c|X|X|X|} \hline
Assumption & \cref{as:A}  & \cref{as:compa-g,as:max-dis-c} & \cref{as:dist} \\ \hline
Type & Steady-state behavior  & Passivity & Steady-state detectability \\  \hline
Purpose & Unique equilibria for constant inputs within the constraint set&  Incremental impedance passivity and closed-loop well-posedness & Strict monotonicity of the steady-state input-to-state map \\ \hline
\end{tabularx}
\caption{Summary of the standing assumptions from Section~\ref{sec:2}.}
\label{table:as}
\end{table*}

\subsection{Steady-state behavior}
\label{sec:steady}
We are given a nonempty closed convex subset $K$ of the input space $U$, which we further assume has nonempty interior.
As mentioned in \cref{sec:intro}, the set $K$ may arise from practical constraints or be comprised of inputs under which we know that system \cref{eq:open-loop-DE} ``behaves well'' at steady state. In \cref{sec:PI} we will introduce a dynamic feedback law for which the input $u$ remains in $K$. 
For now, we make the following assumption.
\begin{as}[Steady-state behavior]
\label{as:A}
The map $A : X \times U \rightrightarrows X$ satisfies:
\begin{enumerate}[label=(\roman*)]
\item \label{it:A-uss} For each $u \in K$, there exists a unique $x \in C$ such that $0 \in A(x, u)$;\footnote{
Recall that the domain of a multivalued map is the set of elements sent to nonempty sets; therefore, $0 \in A(x, u)$ implies $x \in \dom(A(\cdot, u))$.}
\item \label{it:u-inject} 
For each $x \in C$, there exists \emph{at most} one element $u \in K$ such that  $0 \in A(x, u)$.
\end{enumerate}
\end{as}

\Cref{it:A-uss} in \cref{as:A} means that every constant input  $u^\star \in K$ produces a unique steady-state solution $x^\star \in C$. 
On the other hand, it follows from \cref{it:u-inject} that the map $u^\star \mapsto x^\star$ is injective from $K$ to $C$. Thus, throughout this paper, we will often speak of pairs $(x^\star, u^\star)$ of steady-state solution and input, where $u^\star$ uniquely defines $x^\star$ and vice-versa. The ``star'' notation will be convenient to distinguish steady-state, i.e., constant in time trajectories, from time-varying ones.
\begin{defi}[Steady-state pairs]
An element $(x^\star, u^\star) \in C \times K$ is said to be a \emph{steady-state pair} if  $0 \in A(x^\star, u^\star)$.
\end{defi}


\begin{ex}[Finite-dimensional linear systems]
\label{ex:LTI}
Consider, with a slight abuse in notation, the case
$
A(x, u) = Ax + Bu,
$
where $X = C = \R^n$, $U = K = \R^m$, $A \in \R^{n \times n}$ and $B \in \R^{n \times m}$. Then,  \cref{as:A} holds if and only if $\ran(B) \subset \ran(A)$, $\ker(A) = \{0\}$ and $\ker(B) = \{0\}$. In particular, if the uncontrolled system $\dot{x} = Ax$ is \emph{stable}, then $A$ is invertible and \cref{as:A} reduces to $\ker(B) = \{0\}$. In any case, steady-state pairs satisfy $x^\star = -A^{-1}Bu^\star$.
\end{ex}

\begin{ex}[Saturating input]
Let $A(x,u)=-x+{\rm sat}(u)$, with $X=U=\R$, where ${\rm sat}(s)=s$ for $s\in(-1,1)$, ${\rm sat}(s)=-1$ for $s\leq-1$, and ${\rm sat}(s)=1$ for $s\geq1$. In this case, \cref{it:A-uss} in \cref{as:A} holds for any choice of nonempty closed convex set $K\subset\R$, while 
\cref{it:u-inject} holds only for $K\subset(-1,1)$. In the case  $A(x,u)=-{\rm sat}(x)+u$, with $X=U=\R$ and ${\rm sat}$ as above, 
\cref{it:A-uss} holds only for $K\subset(-1,1)$, while \cref{it:u-inject} is satisfied for any choice of $K\subset\R$.
\end{ex}

\subsection{Incremental impedance passivity}
\label{sec:inc-imp}

Let $Y$ be another Hilbert space. We now supplement the control system \cref{eq:open-loop-DE} with an output 
\begin{equation}
\label{eq:output}
y = g(x, u),
\end{equation}
where $g$ is a \emph{single-valued} map from $\dom(g) \subset X \times U$ to $Y$. 
\begin{as}[$A$-compatibility of $g$]\label{as:compa-g} We have $\dom(A) \subset \dom(g)$, and $U = Y$.
\end{as}
\Cref{as:compa-g} guarantees that $g(x, u)$ makes sense whenever $A(x, u)$ is defined, and that passivity (in the sense of Definition~\ref{def:IIP}) is well-defined. The core hypothesis of this work is given next. (We refer the reader to Appendix~\ref{ap:sg} for the definition of maximal dissipativity and relevant related results.)


\begin{as}[Maximal dissipative coupling]\label{as:max-dis-c} The multivalued map $\A : X \times Y \rightrightarrows X \times Y $ defined by
\begin{equation}\label{eq:cal_A}
\A(x, u) \triangleq (A(x, u), -g(x, u))
\end{equation}
is maximal dissipative with $\overline{\dom(\A)} = C \times Y$.
\end{as}

\begin{rem} Continuous contraction semigroups on closed convex subset of Hilbert spaces are \emph{characterized} by maximal dissipative generators; see \cite[Th\'eor\`eme 4.1]{Bre73book} or \cite[Theorem 1.2, Chapter IV]{Bar76book}. Therefore, \cref{as:max-dis-c}
is \emph{equivalent} to $\A$ generating a continuous contraction semigroup on $C \times Y$.
\end{rem}

\begin{rem}
By definition, $\dom(\A) = \dom(A) = \{ (x, u) \in \dom(A(\cdot, u)) \times  Y \} $.
\end{rem}


To 
study system-theoretic properties of the $x$-dynamics, we define solutions to \cref{eq:open-loop-DE} with external input signals.
\begin{defi}[Open-loop solutions]
\label{def:open-loop-DE}
Let $u : \R^+ \to U$ be absolutely continuous.
We say that $x : \R^+ \to X$ is a \emph{strong} solution to 
\cref{eq:open-loop-DE}
if $x$ is {absolutely continuous}, $x(t) \in \dom(A(\cdot, u(t))$ for all $t \geq 0$ and \cref{eq:open-loop-DE} holds a.e.
\end{defi}
\begin{rem}
We do \emph{not} require existence of strong solutions for arbitrary input signals $u$. 
On the other hand, if $(x^\star, u^\star)$ is a steady-state pair, then $x^\star$ (as a constant function of time) is clearly a strong solution to \cref{eq:open-loop-DE} with (constant) input $u^\star$.
\end{rem}
\begin{ex}[Bounded linear control]
\label{ex:bounded-linear-control}
When $A$ has the form $A(x, u) = A(x) + Bu$ where $A : X \rightrightarrows X $ is maximal dissipative and $B : U \to X$ is a continuous linear operator, existence and uniqueness of strong open-loop solutions is guaranteed by \cite[Proposition 3.3]{Bre73book}.
\end{ex}
We now introduce the notion of incremental impedance passivity, which plays a key role in our work. 

\begin{defi}[Incremental impedance passivity]
\label{def:IIP}
We say that the control system \cref{eq:open-loop-DE} with output \cref{eq:output} is \emph{incrementally  impedance passive} if,
given absolutely continuous input signals $u_1, u_2 : \R^+ \to U$, all strong solutions $x_1,x_2$ associated with $u_1,u_2$, respectively, satisfy a.e.
\begin{equation}
\label{eq:inc-imp-pas}
\frac{1}{2} \dt{} \|x_1 - x_2\|^2_X \leq \langle u_1 - u_2, g(x_1, u_1) - g(x_2, u_2) \rangle_{Y}. 
\end{equation}
\end{defi}
This definition naturally extends to the nonlinear setting the concept of infinite-dimensional linear impedance passive system as seen in, e.g., \cite{Sta02,TucWei14}. It is also in line with how finite-dimensional nonlinear incrementally passive systems are defined in, e.g., \cite{JayOrt07}. Note that we restrict ourselves to (in the language of passivity-based control) \emph{quadratic storage functions}.
This restriction is discussed in \cref{sec:conclu}.

\begin{prop} 
\label{prop:indeed}
The control system \cref{eq:open-loop-DE} with output \cref{eq:output} is incrementally impedance passive.
\end{prop}
\begin{proof}

Let $(x_1, u_1) \in \dom(A(\cdot, u_1)) \times Y$  and $(x_2, u_2) \in \dom(A(\cdot, u_2)) \times Y$, i.e.,
$(x_1, u_1),(x_2, u_2)\in\dom(\A)$. By definition of $\A$ being dissipative (see \cref{as:max-dis-c,def:max-dis}),
for all $f_1\in A(x_1, u_1)$ and $f_2 \in A(x_2, u_2)$,
\begin{multline}
\label{eq:sc-IIP}
\langle f_1 - f_2, x_1 - x_2\rangle_X \\ \leq \langle u_1 - u_2, g(x_1, u_1) - g(x_2, u_2)\rangle_Y.
\end{multline}
Let $x_1$ and $x_2$ be two strong solution to \cref{eq:open-loop-DE}. Then, for $i \in \{1, 2\}$, there exists an absolutely continuous input signal $u_i : \R^+ \to U$ such that $\dot{x}_i \in A(x_i, u_i)$ a.e. Absolute continuity of $x_i$ and the chain rule yield $(1/2)(\d / \d t)\|x_1 - x_2\|^2_X = \langle \dot{x}_1 - \dot{x}_2, x_1 - x_2 \rangle_X$ a.e. Using \cref{eq:sc-IIP} we obtain \cref{eq:inc-imp-pas}.
\end{proof}
\begin{rem}
\label{rem:mono-gain}
Incremental impedance passivity implies that the steady-state input-output map $u^\star \in K \mapsto g(x^\star, u^\star) \in Y$ is \emph{monotone}: if $(x^\star, u^\star),(x^\dagger, u^\dagger)$
are steady-state pairs 
then
\begin{equation}
\langle u^\star - u^\dagger, g(x^\star, u^\star) - g(x^\dagger, u^\dagger) \rangle_Y \geq 0.
\end{equation}
\end{rem}

Another consequence of incremental impedance passivity is that fixed initial data and input signal produce at most one solution, as shown in the next lemma.

\begin{lemma}[Uniqueness of open-loop solutions]
\label{lem:uniq-ol}
Given $u : \R^+ \to U$ absolutely continuous and $x_0 \in X$,
there is at most one strong solution to \cref{eq:open-loop-DE}
that satisfies 
$x(0) = x_0$.
\end{lemma}

\begin{proof} Any two strong solutions $x_1, x_2$ to \cref{eq:open-loop-DE} with the same input $u$ must satisfy $(\d / \d t)\|x_1 - x_2\|^2_X \leq 0$ a.e. Since $\|x_1 -x_2\|^2_X$ is absolutely continuous, it follows that $\|x_1(t) - x_2(t)\|^2_X \leq \|x_1(0) - x_2(0)\|^2_X = 0$ for all $t \geq 0$.
\end{proof}

\subsection{Energy dissipation and steady-state detectability}
\label{sec:detec}

Next, we associate to each pair of strong solutions a notion of ``internal energy dissipation''. This will help formulating and discussing open-loop stability as well as special detectability properties of the plant defined by \cref{eq:open-loop-DE,eq:output}.
\begin{defi}[Internal dissipation]
\label{def:h}
Let $x_1$ and $x_2$ be strong solutions to \cref{eq:open-loop-DE} 
with respective input signals $u_1$ and $u_2$. We associate with the pairs $(x_1, u_1)$ and $(x_2, u_2)$ a function $t \mapsto h(x_1, u_1; x_2, u_2)(t)$  defined a.e.\footnote{
Strictly speaking, each $h(x_1, u_1; x_2, u_2)$ is defined on the subset of $(0, +\infty)$ where both $x_1$ and $x_2$ are differentiable, and the complement of this set has Lebesgue measure zero.
} in $(0, +\infty)$ by
\begin{multline}\label{eq:h}
h(x_1, u_1; x_2, u_2) \triangleq \langle u_1 - u_2, g(x_1, u_1) - g(x_2, u_2)\rangle_Y \\ 
- \frac{1}{2}\dt{} \|x_1 - x_2\|^2_X \quad \mbox{a.e.}
\end{multline}
\end{defi}
The function $h$ allows us to write \emph{exact} incremental energy balances. Incremental impedance passivity \cref{eq:inc-imp-pas} implies that $h \geq 0$ a.e. Furthermore, it follows from \cref{eq:h} that pairs of strong solutions satisfy the incremental energy equality
\begin{multline}\label{eq:inc-imp-pass-h}
\frac{1}{2} \dt{} \|x_1 - x_2\|^2_X + h(x_1, u_1; x_2, u_2) =  \\ 
\langle u_1 - u_2, g(x_1, u_1) - g(x_2, u_2)\rangle_Y \quad \mbox{a.e.}
\end{multline}
{
\begin{ex}[Damped linear system] 
\label{eq:damped-linear} In the setting of \cref{ex:LTI}, suppose that the output is given by $y = B^\ast x$ and $A = S - DD^\ast$ where $S \in \R^{n\times n}$ is skew-symmetric and $D \in \R^{n \times m}$.
Then, given a strong solution $x$ with input signal $u$, $h(x, u; 0, 0) = \|D^\ast x\|^2$ a.e.
\end{ex}
}
As suggested by \cref{eq:damped-linear}, we may think of $h$ as the squared norm of some ``virtual output'' involved in an internal dissipative feedback mechanism of the $x$-dynamics. In our approach, further knowledge of $h$ for \emph{arbitrary} open-loop solutions is not required.
On the other hand, for steady-state solutions, $h$ has a more explicit form, as shown below. 

\begin{lemma}[$h$ for steady states]
\label{lem:h-semi} Let $(x^\star, u^\star)$ and $(x^\dagger, u^\dagger)$ be two steady-state pairs. Then,
\begin{equation}\label{eq:h-steady}
h(x^\star, u^\star; x^\dagger, u^\dagger) = \langle u^\star - u^\dagger, g(x^\star, u^\star) - g(x^\dagger, u^\dagger)\rangle_Y.
\end{equation}
\end{lemma}

\begin{proof}
The constant function $x^\star$ (resp.\ $x^\dagger$) is a strong solution to \cref{eq:open-loop-DE} with input $u^\star$ (resp.\ $u^\dagger$). \hspace{-1mm} Thus, \cref{eq:h} implies \cref{eq:h-steady}. 
\end{proof}

We now introduce an assumption regarding ``distinguishability'' of the steady states of the plant \cref{eq:open-loop-DE}--\cref{eq:output}.
\begin{as}[Distinguishability of steady states]
\label{as:dist} 
Let $(x^\star, u^\star)$ and  $(x^\dagger, u^\dagger)$ be steady-state pairs.
If $h(x^\star,u^\star; x^\dagger,u^\dagger) = 0$ then $x^\star = x^\dagger$.
\end{as}
\begin{rem}\label{rem:dist} In light of \cref{rem:mono-gain} and \cref{lem:h-semi},
 \cref{as:dist} can be reformulated as \emph{strict} monotonicity of the steady-state input-output map $u^\star\in K \mapsto g(x^\star, u^\star)\in Y$, as $x^\star=x^\dagger$ implies $u^
 \star=u^\dagger$ by \cref{as:A}. In particular, given any $r \in Y$, there exists \emph{at most} one steady-state pair $(x^\star, u^\star)$ such that $g(x^\star, u^\star) = r$. Requiring strict monotonicity of the steady-state input-output map is a standard assumption in, e.g., low-gain integral control, see \cite{DesLin85,LorWei22,LorWei23,Sim20}.
\end{rem}

\begin{ex}
In the setting of \cref{ex:LTI}, \cref{as:dist} amounts to injectivity of the transfer function at $0$. A typical situation where \cref{as:dist} fails to hold is when the system has a second-order form with ``collocated'' input and output:
$
\ddot{x} + Lx = Bu$,  $y = B^\ast \dot{x}
$. Indeed, at steady state the output of such a plant is always zero. In these cases the transfer function has a \emph{transmission zero} at zero, contradicting the \emph{non-resonance condition}, which is a well-known necessary condition for solving the output regulation problem \cite{IsiMar03}.
\end{ex}

\section{Set-point  tracking problem and closed-loop} 
\label{sec:PI}

We begin by describing
the set-point tracking problem and by recalling the projected integral controller from 
\cite{LorWei23},
which we use in place of a classical integrator. This allows the closed-loop system to operate only in a region of the (extended) state space, where the plant enjoys suitable steady-state properties; see \cref{as:A}.
Then, we establish the well-posedness of the proposed closed-loop system, and we derive additional closed-loop properties that will be instrumental in carrying out the stability analysis, presented later in \cref{sec:4}.

\subsection{Integral action and projected dynamics}

Given a constant reference signal $r \in Y$, our goal is to find a control $u$ that guarantees that solutions $x$ to \cref{eq:open-loop-DE} remain bounded and solve the \textit{set-point tracking problem}, i.e.,
\begin{equation}
\label{eq:conv-output}
g(x(t), u(t)) \to r, \quad t \to + \infty.
\end{equation}
In addition, we require that the control $u$ satify the convex constraint $u \in K$.
The standard approach is to augment the $x$-dynamics with an output error integrator governed by $\dot{z} = r - g(x, u)$. Then, passivity-based control suggests closing the loop via the very simple output feedback $u = z$; see \Cref{lem:cl-inc-en}.  Integral action is motivated by the property that at \emph{any} equilibrium $(x^\star,z^\star)$ of the extended $(x, z)$-dynamics we must have $g(x^\star) = r$. To enforce the convex constraint on the control,
we will use a \emph{projected} integral controller instead. 

\begin{defi}[Normal cone]
The normal cone $N_K(z)$ to the convex set $K$ at $z \in K$ is defined as
\begin{equation}
N_K(z) \triangleq \{  w \in Y : \langle w, z - y \rangle_Y \geq 0 ~\mbox{for all}~y \in K \}.
\end{equation}
If $z \not \in K$, we let $N_K(z) \triangleq \emptyset$.
\end{defi}
We have $\dom(N_K) = K$ and also 
$N_K(z) = \{0\}$ for all $z \in K^\circ$.
Furthermore, the set-valued map $N_K : Y \rightrightarrows Y$ is maximal monotone, i.e., $-N_K$ is maximal dissipative; see, e.g., \cite[Corollary~12.18]{RocWet09} in the finite-dimensional setting or \cite[Theorem A]{Roc70b} in general.
Given $r \in Y$, we define the \emph{ 
projected integral controller} as a system described by
\begin{equation}
\label{eq:proj-int}
\dot{z} \in r - g(x, z) - N_K(z).
\end{equation}
Roughly speaking, when $z$ lies in $K^\circ$, $N_K(z) = 0$ and $z$ simply integrates the output error $r - g(x)$; however, whenever $z$ hits the boundary $\partial K$, a force $f \in -N_K(z)$ prevents $z$ from leaving the set $K$. This force $f$ can be determined studying the \textit{principal section} associated to \cref{eq:proj-int}; see \cref{def:principal}. 


\begin{rem}
The projected integral controller \cref{eq:proj-int} can be equivalently formulated, see \cite{BroDan06}, as $\dot{z}=\Pi_K(z,r - g(x, z))$, where $\Pi_K(z,w)\triangleq\arg\min_{v\in T_K(z)}\|w-v\|$, with $T_K\triangleq\{  v \in Y : \langle w, v \rangle_Y \leq 0 ~\mbox{for all}~w \in N_K(z) \}$ being the tangent cone to the set $K$ at $z\in K$. This latter formulation can be approximated (and, thus, implemented) using an anti-windup design, as shown in \cite{HauDor20}. A similar design is proposed in \cite{AstMar22}, for linear single-input single-output operator semigroups with input saturation. In particular, the controller they implement approximates \cref{eq:proj-int} by means of a \textit{deadzone} nonlinearity.
\end{rem}

\begin{rem}[Convergence of the output]
\Cref{eq:conv-output} is to be understood in a formal way. For infinite-dimensional systems, the output $g(x, u)$ may not be defined in a pointwise,  classical sense for all solutions. Our results in \cref{sec:tracking} will guarantee convergence to an equilibrium $x^\star \in \dom(g)$ where $g(x^\star, u^\star) = r$. The matter of actual convergence of the output, which is nontrivial unless $g$ is a continuous function of $X\times Y$, requires case-by-case examination; see also the discussion of \cite[Section 3.4]{VanBri23}.  The case studies in \cref{sec:5} feature different types of output convergence obtained by exploiting additional structure of the systems under consideration.
\end{rem}


\subsection{Closed-loop well-posedness and additional properties}

Let $r \in Y$. The differential inclusions consisting of \cref{eq:open-loop-DE}--\cref{eq:output} in closed loop with \cref{eq:proj-int} have the form:
\begin{subequations}
\label{eq:closed-loop}
\begin{align}
&\dot{x} \in A(x, z), \\
\label{eq:PIC}
&\dot{z} \in  r - g(x, z) - N_K(z).
\end{align}
\end{subequations}
Define the map $\A_r : X \times Y \rightrightarrows X \times Y$ as
\begin{align}
\label{eq:cal_A_r}
\A_r(x, z) &\triangleq (A(x, z), - N_K(z) + r - g(x, u)), \\
\dom(\A_r) &=  \{ (x, z) \in C \times K : x \in  \dom(A(\cdot, z)) \}. \notag
\end{align}
Well-posedness of the closed-loop  inclusions \cref{eq:closed-loop} is understood in terms of existence and uniqueness of strong (resp.\ generalized) solutions to the abstract evolution equation $(\dot{x}, \dot{z}) \in \A_r(x, z)$ for all initial data $(x_0, z_0)$ in $\dom(\A_r)$ (resp.\ $C \times K$). The precise meaning of strong and generalized solutions is given in \cref{def:str-sol,def:gen-sol}.


\begin{prop}
\label{prop:max} 
The  operator $\A_r$ is maximal dissipative. 
\end{prop}
\begin{proof}
Recall the maximal dissipative operator $\A : X \times Y \rightrightarrows X \times Y$ from \cref{sec:inc-imp}, with $\dom(\A) = \{ (x,z) \in \dom(A(\cdot, z)) \times Y\}$, and let
\begin{equation}
\K_r(x,z) \triangleq (0, r - N_K(z)), \quad \dom(\K_r) = X \times K
\end{equation}
Clearly, $\dom(\A) \cap \dom(\K_r) = \dom(\A_r)$  and $\A_r = \A + \K_r$. Note that the operator $\A$ represents the original $x$-system in closed-loop with the unconstrained output integrator, while $\K_r$ 
is a trivial extension of $r -N_K$ over $X \times Y$. We observe that
\begin{multline}
    \label{eq:int-rock}
\dom(\A) \cap (\dom(\K_r))^\circ \\ = \{ (x, z) \in C \times K^\circ : x \in \dom(A(\cdot, z)) \}.
\end{multline}
The set in \cref{eq:int-rock} is nonempty: $K^\circ$ is nonempty (by assumption) and, for any $z \in K^\circ$, $\dom(A(\cdot, z))$ is nonempty since, in particular, it contains  a steady state (\cref{as:A}).
Hence, by virtue of
\cite[Theorem 1]{Roc70}, for $\A_r$ to be maximal dissipative, it suffices that $\A$ and $\K_r$ be maximal dissipative. $-N_K$ and, thus, $r - N_K$ are maximal dissipative; $0$ is trivially maximal dissipative and so is $\K_r$ due to its diagonal structure.
Finally, $\A$ is assumed maximal dissipative. 
\end{proof}


\begin{coro}[Closed-loop semigroup generation] 
\label{cor:cl_sem_gen}
There exists a continuous semigroup of contractions $\sg{\S^r}$ on $C \times K$ such that,
given initial data $(x_0, z_0) \in C \times K$,  $(x, z)$ defined by $(x(t), z(t)) \triangleq \S_t^r(x_0, z_0)$ for $t \geq 0$ is the unique generalized solution of the closed-loop inclusions \cref{eq:closed-loop}. If in addition $(x_0, z_0) \in \dom(\A_r)$, then $(x, z)$ is the unique strong solution of \cref{eq:closed-loop}. 
\end{coro}
\begin{proof}
It is straightforward that $\overline{\dom(\A_r)} = C \times K$. To conclude the proof, apply \cref{theo:abstract-cauchy,coro:generation}.
\end{proof}
\begin{rem}
Let $(x, z)$ be a strong solution to the closed-loop inclusions \cref{eq:closed-loop}. Then $x$ is an open-loop strong solution to \cref{eq:open-loop-DE} with input signal $z$, in the sense of \cref{def:open-loop-DE}.
\end{rem}

Next, we derive a closed-loop incremental energy identity.

\begin{prop}[Closed-loop incremental energy inequality]
\label{lem:cl-inc-en}
Let $(x_1, z_1)$ and $(x_2, z_2)$ be two strong solutions to the closed-loop inclusions \cref{eq:closed-loop}. Then,
\begin{multline}
\label{eq:cl-inc-en}
\frac{1}{2} \dt{}  \{ \|x_1 - x_2\|^2_X + \|z_1 - z_2\|^2_Y  \} \\ + h(x_1, z_1;  x_2, z_2) \leq 0 \quad \mbox{a.e.}
\end{multline}
\end{prop}
\begin{proof}
\Cref{eq:PIC} implies that $\dot{z}_i  - r + g(x_i,z_i) \in - N_K(z_i)$. Since $-N_K : Y \rightrightarrows Y$ is dissipative, $\langle \dot{z}_1 -r + g(x_1,z_1) -\dot{z}_2 + r - g(x_2,z_2), z_1 - z_2 \rangle_Y \leq 0$. Thus,
\begin{multline}
\label{eq:z-diff}
\frac{1}{2}\dt{}\|z_1 - z_2\|^2_Y = \langle \dot{z}_1 - \dot{z}_2, z_1 - z_2\rangle_Y \\ \leq - \langle g(x_1,z_1) - g(x_2,z_2), z_1 - z_2\rangle_Y \quad \mbox{a.e.}, 
\end{multline}
where the first equality comes from the chain rule. Summing \cref{eq:z-diff} with the identity \cref{eq:inc-imp-pass-h} stemming from incremental impedance passivity finally yields \cref{eq:cl-inc-en}.
\end{proof}

\begin{coro}[Uniqueness of closed-loop equilibrium]
\label{coro:uni-CL} The closed-loop system \cref{eq:closed-loop} possesses at most one equilibrium.
\end{coro}
\begin{proof}
Let $(x^\star, u^\star), (x^\dagger, u^\dagger) \in \dom(\A_r)$ be two equilibria of \cref{eq:closed-loop}. By \cref{eq:cl-inc-en} and $h\geq0$, $h(x^\star, u^\star;  x^\dagger, u^\dagger) = 0$, which in view of \cref{as:dist} implies $x^\star = x^\dagger$. Part \cref{it:u-inject} of \cref{as:A} then directly gives $u^\star = u^\dagger$.
\end{proof}

Finally, we provide sufficient conditions for compactness of the resolvents $(\lambda \id - \A_r)^{-1}$, $\lambda > 0$. This will be 
needed to obtain relative compactness of closed-loop state trajectories. 
\begin{as}[Sufficient conditions for compactness]
\label{as:compact}
The following hold:
\begin{enumerate}[label=(\roman*)]
\item The space $Y$ is finite-dimensional;
\item \label{it:unif-comp}There exists $\lambda > 0$ such that the map
\begin{equation}  
(f, u)   \mapsto (\lambda \id - A(\cdot, u))^{-1}(f) 
\end{equation}
is compact\footnote{
This map is possibly multivalued. By \emph{compact} we mean that the set of all $x$ satisfying $\lambda x - A(x, u) \owns f$, with $f$ and $u$ varying within fixed bounded subsets of $X$ and $U$, is relatively compact in $X$.
} from $X \times U$ into $X$.
\end{enumerate}
\end{as}
\begin{rem}\label{rem:comp_finit_dim}
If $X$ and $Y$ are both finite-dimensional then \cref{as:compact} is automatically satisfied.
\end{rem}
\begin{ex}
In the setting of \cref{ex:bounded-linear-control}, \cref{as:compact}\cref{it:unif-comp} reduces to compactness of 
$(\lambda \id - A)^{-1}$, $\lambda > 0$.
\end{ex}

\begin{prop}
\label{prop:compact}
The resolvents $(\lambda \id - \A_r)^{-1}$, $\lambda  >0$, are compact from $X \times Y$ into itself.
\end{prop}
\begin{proof}
Since the closed-loop generator $\A_r$ is maximal dissipative, for all $\lambda > 0$, $(\lambda \id - \mathcal{A}_r)^{-1}$ is well-defined and  a contraction on $X \times Y$. Furthermore, it follows from the resolvent identity \cite{Bre73book} that $\lambda( \lambda \id - \A_r)^{-1}$ can be written as
\begin{equation}
\mu(\mu \id - (\lambda \mu^{-1} \id + \lambda(1 - \lambda \mu^{-1})(\lambda \id - \A_r)^{-1}))^{-1},
\end{equation}
for any $\lambda, \mu > 0$. Thus, the compactness property is independent of the choice of $\lambda > 0$, and it is enough to prove compactness for $\lambda > 0$  
as in \cref{as:compact}.
For any $(f, y) \in X \times Y$, there exists a unique $(x, z) \in \dom(\A_r)$:
\begin{subequations}
\label{eq:max-cl}
\begin{align}
\label{eq:max-cl-x}
&\lambda x - A(x, z) \owns f, \\
&\lambda z  + N_K(z) \owns r + y - g(x,z).
\end{align}
\end{subequations}
Let $\B$ be a fixed bounded subset of $X \times Y$. By contractivity, $(\lambda \id - \A_r)^{-1}\B$ is bounded in $X \times Y$. 
In particular, when $(f, y)$ varies in $\B$, the corresponding $z$ remain bounded in $Y$, which, by \cref{as:compact} is finite dimensional. Thus, all such $z$ lie in a relatively compact subset of $Y$. Moreover, by \cref{it:unif-comp} in \cref{as:compact}, the set of all $x$ solving \cref{eq:max-cl-x} when $(f, z)$ is in a fixed bounded set of $X \times Y$ is relatively compact in $X$. Therefore, $(\lambda \id - \A_r)^{-1}\B$ is relatively compact in $X \times Y$.
\end{proof}

\section{Output tracking and stability analysis}
\label{sec:4}
\label{sec:tracking}



In this section we present our main results.
\Cref{theo:strict-tracking} addresses a special class of systems, which we call \emph{strictly} output incrementally impedance passive, and provides sufficient conditions for output tracking in terms of open-loop assumptions only. In contrast, \Cref{prop:stab-detec} is more general, but requires a closed-loop detectability property.

\begin{defi}[Feasible references]
\label{def:feasible}
A reference $r \in Y$ is said to be \emph{feasible} if there exists
$u^\star \in K$
such that the\footnote{As stated in \cref{rem:dist}, such a pair is unique (in $C \times K$).} steady-state pair $(x^\star, u^\star)$ satisfies $g(x^\star,u^\star) = r$.
\end{defi}
\begin{lemma}
\label{lem:feasible}
If a reference $r \in Y$ is feasible, then the corresponding steady-state pair $(x^\star, u^\star)$ is 
an equilibrium of the closed loop \cref{eq:closed-loop}. 
\end{lemma}
\begin{proof}
By \cref{lem:equilibria}, it suffices to prove that $0 \in \A_r(x^\star, u^\star)$.
It is clear that $(x^\star, u^\star) \in \dom(\A_r)$ and we already have $0 \in A(x^\star, u^\star)$. By definition,
$g(x^\star, u^\star) = r$. Thus the second ``coordinate'' of $\A_r(x^\star, u^\star)$ is the set $-N_K(u^\star)$, which also contains $0$ since $u^\star \in K$.
\end{proof}



\subsection{Set-point tracking for strictly output passive systems}
The following definition is inspired by the notion of strict output passivity from, e.g., \cite{VdS00}, \cite{JayWei09}, \cite[Definition~6.3]{Kha02}. 

\begin{defi}[Strict output incremental passivity]
\label{def:SOIIP}
The open-loop system \cref{eq:open-loop-DE}--\cref{eq:output} is said to be \emph{strictly} output incrementally  passive if, given pairs $(x_1, u_1)$ and $(x_2, u_2)$ of strong solutions to \cref{eq:open-loop-DE} and corresponding inputs, $h(x_1, u_1; x_2, u_2) = 0$ a.e.\ implies $g(x_1, u_1) = g(x_2,u_2)$ a.e.
\end{defi}

\begin{rem}\label{rem:SOIIP} 
The incremental version of the strict output passivity property defined in \cite{Kha02} would read as follows: there exists $k > 0$ such that, for any pairs 
given pairs $(x_1, u_1)$ and $(x_2, u_2)$ of strong solutions to \cref{eq:open-loop-DE} and corresponding inputs, 
\begin{multline}\label{eq:SOIIP}
\frac{1}{2}\dt{} \|x_1 - x_2\|^2_X \leq \langle u_1 - u_2, g(x_1, u_1) - g(x_2,u_2)\rangle_Y \\ - k \|g(x_1, u_1) - g(x_2, u_2)\|^2_Y \quad \mbox{a.e.}, 
\end{multline}
for some $k>0$.
If \cref{eq:SOIIP} holds, then $h(x_1, u_1; x_2, u_2)\geq k\|g(x_1, u_1)-g(x_2, u_2)\|^2$ a.e. by \cref{eq:h}, and the open-loop system \cref{eq:open-loop-DE}--\cref{eq:output} meets the requirement of \cref{def:SOIIP}.
\end{rem}

\begin{ex}
Consider the system $\dot{x} = -x^{1/3} + u$, $g(x,u) = x$, with $X = Y = U = \R$. For such a system, $h$ is given by 
\begin{equation}
h(x_1, u_1; x_2, u_2) = (x_1^{1/3} - x_2^{1/3})(x_1 - x_2),
\end{equation}
and the property of strict output impedance passivity in the sense of \cref{def:SOIIP} is satisfied. However, \cref{eq:SOIIP} fails to hold due to the fact that $x^{1/3}x|x|^{-2} \to 0$ when $|x| \to + \infty$.
\end{ex}

We introduce next the notion of \emph{$h$-detectable} steady states. 
\begin{defi}[$h$-detectable steady states]
\label{as:OL-detec}
Let $(x^\star, u^\star)$ be a steady-state pair. We say that $(x^\star, u^\star)$ is \emph{$h$-detectable} if, given any strong  solution $x$ to \cref{eq:open-loop-DE} with constant input $u^\star$,
$h(x, u^\star; x^\star, u^\star) = 0$ a.e.\ implies $x(t) = x^\star$ for all $t \geq 0$.
\end{defi}
\begin{rem}
In the setting of \cref{as:OL-detec}, by the open-loop  incremental energy equality \cref{eq:inc-imp-pass-h}, the condition $x(t) = x^\star$ for all $t\geq 0$ is \emph{equivalent} to $x(t) \to x^\star$ as $t \to + \infty$, the latter being closer to the standard detectability terminology.
\end{rem}

We need to discuss \cref{as:OL-detec},  as a detectability property stated without explicit mention of an output may seem strange at first glance. Setting $u = u^\star$ in \cref{eq:open-loop-DE} and viewing $h(x, u^\star; x^\star, u^\star)$ as the (squared) norm of some ``virtual output'', \cref{as:OL-detec} reduces to {zero-state detectability} of the plant, as defined in, e.g., \cite{ByrIsi91,SepJan12}. Furthermore, under the additional assumption that the system \cref{eq:open-loop-DE}--\cref{eq:output} is strictly output incrementally impedance passive in the sense of \cref{def:SOIIP}, \cref{as:OL-detec} amounts to zero-state detectability of the plant with respect to the actual output $y = g(x, u)$.

We are now ready to state our first output tracking theorem, whose proof is postponed to \cref{sec:closed-loop-cond}.
\begin{theo}[Output tracking for strictly output passive systems]
\label{theo:strict-tracking}
Suppose that
\begin{enumerate}[label=(\roman*)]
    \item The open-loop system \cref{eq:open-loop-DE}--\cref{eq:output} is strictly {output} incrementally passive;
    \item All steady-state pairs are $h$-detectable.
\end{enumerate}
Let $r \in Y$ be feasible and let $(x^\star, u^\star) \in \dom(A(\cdot, u^\star)) \times K$ be the unique (among elements of $C \times K$) pair of steady-state solution and input such that $g(x^\star, u^\star) = r$. Then, for all initial data $(x_0, z_0) \in C \times K$, the corresponding solution $(x, z) \in \C(\R^+, X \times Y)$ of the closed loop \cref{eq:closed-loop} satisfies
\begin{equation}
\label{eq:ASb}
(x(t), z(t)) \to (x^\star, u^\star) \quad \mbox{in}~X \times Y, \quad t \to + \infty.
\end{equation}
Moreover, for any $r\in Y$ and for all initial data $(x_0, z_0) \in C \times K$, the convex constraint $z(t)\in K$ holds for all $t\geq0$.
\end{theo}

\begin{rem}
\cref{theo:strict-tracking} requires that \emph{all} steady-state pairs be $h$-detectable, not just the closed-loop equilibrium's one. 
\end{rem}

The hypotheses of \cref{theo:strict-tracking} guarantee that the equilibrium $(x^\star, u^\star)$ is ``closed-loop $h$-detectable'' in the sense of \cref{def:cl-detec}. In the single-valued, finite-dimensional context, similar ideas are found in \cite{Jay05,JayWei09}, where it is 
shown that, under suitable assumptions, open-loop detectability properties are inherited by the closed loop. Another possibility, as, e.g., in \cite[Proposition~3]{JayOrt07}, is to assume \emph{a priori} that closed-loop equilibria possess desired detectability properties. This is the approach taken in \cref{prop:stab-detec}. 



\subsection{Closed-loop sufficient condition for output tracking}

\label{sec:closed-loop-cond}

In this section we introduce the notion of $h$-detectability for closed-loop equilibria and we show that such equilibria are asymptotically stable. With these tools at hand, we will be able to prove \cref{theo:strict-tracking} as a consequence of \cref{prop:stab-detec}.


\begin{defi}[Closed-loop $h$-detectability] \label{def:cl-detec}
Let $r \in Y$ and let $(x^\star, u^\star) \in \dom(\A_r)$ be
an equilibrium  of the closed-loop system \cref{eq:closed-loop}. 
We say that $(x^\star, u^\star)$ is \emph{$h$-detectable} if, given any strong solution $(x, z)$ to \cref{eq:closed-loop}, $h(x, z;  x^\star, u^\star) = 0$ a.e.\ implies $x(t) \to x^\star$ in $X$ as $t \to + \infty$.
\end{defi}

Our second main result is presented next.
When strict output incremental impedance passivity does not hold, but closed-loop equilibria are nonetheless detectable, \cref{prop:stab-detec} guarantees closed-loop stability and set-point output tracking. 
Depending on the system under consideration, one may use \cref{theo:strict-tracking} or \cref{prop:stab-detec}.

 \begin{theo}[Detectable equilibria are asymptotically stable] 
 \label{prop:stab-detec} 
 Let $r \in Y$
 and let $(x^\star, u^\star) \in \dom(\A_r)$ be an $h$-detectable equilibrium of the closed loop \cref{eq:closed-loop}.
Then for all initial data $(x_0, z_0) \in C \times K$ the corresponding solution $(x, z) \in \C(\R^+, X \times Y)$ to \cref{eq:closed-loop} satisfies
\begin{equation}
\label{eq:AS}
(x(t), z(t)) \to (x^\star, u^\star) \quad \mbox{in}~X \times Y, \quad t \to + \infty.
\end{equation}
Moreover, for any $r\in Y$ and for all initial data $(x_0, z_0) \in C \times K$, the convex constraint $z(t)\in K$ holds for all $t\geq0$.
\end{theo}
\begin{proof}
The proof is divided into several steps.

\emph{Step 1: Preliminary remark.} Since the closed-loop semigroup $\sg{\S^r}$ associated to \cref{eq:closed-loop} is contractive and $\dom(\A_r)$ is dense in $C \times K$, it suffices to prove \cref{eq:AS} for initial data $(x_0, z_0) \in \dom(\A_r)$.


\emph{Step 2: LaSalle argument.} We now recall a standard line of arguments in the spirit of LaSalle's invariance principle. Fix $(x_0, z_0) \in \dom(\A_r)$. We have shown in \cref{prop:compact} that the closed-loop generator $\A_r$ has compact resolvent. Besides, $0 \in \ran(\A_r)$ as $(x^\star, u^\star) \in \dom(\A_r)$ is assumed to be an equilibrium of the closed loop \cref{eq:closed-loop}. By \cite[Theorem 3]{DafSle73} this implies that the semi-orbit $\{(x(t), z(t)), t \geq 0\}$ is a relatively compact subset of $X \times Y$. Therefore, the $\omega$-limit set 
\begin{equation}
\omega(w_0, z_0) \triangleq \bigcap_{s \geq 0} \overline{\bigcup_{t \geq s}\{\S_t^r(x_0, z_0)\}},
\end{equation}
where the closure is taken in $X \times Y$,
is a nonempty compact subset of $C \times K$ by \cite[Th\'eor\`eme 1.1.8]{Har91book}, satisfying:
\begin{enumerate}[nosep]
    \item (\emph{Invariance.}) $\S_t^r\omega(x_0, z_0) = \omega(x_0, z_0)$ for all $t \geq 0$;
    \item (\emph{Attractivity.}) $\dist((x(t), z(t)), \omega(x_0, z_0)) \to 0$ as $t \to + \infty$.
\end{enumerate}
Using the sequential characterization of $\omega$-limit sets and \cite[Lemma 2.3]{CraPaz69}, we can show that $\omega(x_0, z_0) \subset \dom(\A_r)$. To obtain \cref{eq:AS}, it suffices to show that $\omega(x_0, z_0) = \{(x^\star, u^\star)\}$.

\emph{Step 3: Closed-loop solutions on the limit set.} We may pick an arbitrary element $(\tilde{x}_0, \tilde{z}_0) \in \omega(x_0, z_0)$. Let $(\tilde{x}(t), \tilde{z}(t)) \triangleq \S_t^r(\tilde{x}_0, \tilde{z}_0)$ for $t \geq 0$. Since $(\tilde{x}_0, \tilde{z}_0) \in \dom(\A_r)$, $(\tilde{x}, \tilde{z})$ is a strong solution to the closed-loop inclusions \cref{eq:closed-loop}. Now, recall that $(x^\star, u^\star)$ is a  fixed point of $\sg{\S^r}$. By virtue of \cite[Theorem 1]{DafSle73}, $\omega(x_0, z_0)$ must lie on a sphere centered around $(x^\star, u^\star)$, which means that $(1/2) (\d / \d t)\{ \|\tilde{x} - x^\star\|^2_X + \|\tilde{z} - z^\star\|^2_Y  \} = 0$. In view of the incremental closed-loop energy inequality \cref{eq:cl-inc-en}, $h(\tilde{x}, \tilde{z}; x^\star, u^\star) \leq 0$, but we also know that, by definition, $h(\tilde{x}, \tilde{z}; x^\star, u^\star) \geq 0$, resulting in
\begin{equation}
\label{eq:h-vanishes}
h(\tilde{x}, \tilde{z}; x^\star, u^\star) = 0 \quad \mbox{a.e.\ in}~(0, +\infty).
\end{equation}
Since $(x^\star, u^\star)$ is assumed detectable in the sense of \cref{def:cl-detec}, \cref{eq:h-vanishes} implies that
$\tilde{x}(t) \to x^\star$ in $X$ as $t \to + \infty$. 

\emph{Step 4: Covering argument.} Next, we wish to prove that $\omega(x_0, z_0) \subset \{x^\star\} \times K$. Let $\eps > 0$. Because $\omega(x_0, z_0)$ is compact, it possesses a finite $\eps$-covering: there exist $N \in \N$ and elements $(\tilde{x}^i_0, \tilde{z}^i_0) \in \omega(x_0, z_0)$, $i =1, \dots, N$, such that
\begin{equation}
\label{eq:covering}
\omega(x_0, z_0) \subset \bigcup_{i=1}^{N} \B_{X\times Y}((\tilde{x}^i_0, \tilde{z}^i_0), \eps).
\end{equation}
For  $i = 1, \dots, N$ and $t \geq 0$, write $(\tilde{x}^i(t), \tilde{z}^i(t)) \triangleq \S_t^r(\tilde{x}^i_0, \tilde{z}^i_0)$. By the previous step, each $\tilde{x}^i$ satisfies 
$\tilde{x}^i(t) \to x^\star$ in $X$ as $t \to + \infty$. 
Thus, for each $i$ there exists $\tau_i \geq 0$ such that $\tilde{x}^i(t) \in \B_X(x^\star, \eps)$ for all $t \geq \tau_i$. Now, set $\tau \triangleq \max \{ \tau_i, i = 1, \dots, N\}$. Pick an arbitrary $(\tilde{x}_0, \tilde{z}_0) \in \omega(x_0, z_0)$ and let $(\tilde{x}(t), \tilde{z}(t)) \triangleq \S_t^r(\tilde{x}_0, \tilde{z}_0)$ for $t \geq 0$. By \cref{eq:covering}, we can find an $i$ such that $\|\tilde{x}_0 - \tilde{x}_0^i\|_X^2 +  \|\tilde{z}_0 - \tilde{z}_0^i\|_Y^2 \leq \eps^2$. Then, for all $t \geq \tau$,
\begin{multline*}\label{eq:ineq-tilde-star}
\|\tilde{x}(t) - x^\star\|^2_X \leq 2\|\tilde{x}(t) - \tilde{x}^i(t)\|^2_X + 2\|\tilde{x}^i(t) - x^\star\|^2_X \\ 
+ 2\|\tilde{z}^i(t) - \tilde{z}(t)\|^2_Y \leq 2 \eps^2 + 2\|\tilde{x}_0 - \tilde{x}_0^i\|_X^2 +  2\|\tilde{z}_0 - \tilde{z}_0^i\|_Y^2 \leq 4\eps^2,
\end{multline*}
where we also used the contraction property of the closed-loop semigroup $\sg{\S^r}$. As the data $(\tilde{x}_0, \tilde{z}_0)$ are chosen arbitrarily in $\omega(x_0, z_0)$, the above shows that $\S_t^r\omega(x_0, z_0) \subset \B_{X}(x^\star, 2\eps) \times K$ for all $t \geq \tau$. But because $\S_t^r\omega(x_0, z_0) = \omega(x_0, z_0)$ for all $t \geq 0$, we, in fact, have $\omega(x_0, z_0) \subset \B_{X}(x^\star, 2\eps) \times K$. Now, recall that $\eps > 0$ was chosen arbitrarily as well. Thus, by letting $\eps \to 0$, we obtain
\begin{equation}
\label{eq:str-omega}
\omega(x_0, z_0) \subset \{x^\star\} \times K.
\end{equation}

\emph{Step 5: Conclusion.} Let $(\tilde{x}_0, \tilde{z}_0) \in \omega(x_0, z_0)$ and $(\tilde{x}(t), \tilde{z}(t)) \triangleq \S_t^r(\tilde{x}_0, \tilde{z}_0)$ for $t \geq 0$. According to \cref{eq:str-omega}, $\tilde{x}(t) = x^\star$ for all $t \geq 0$ and in particular $\dot{x} = 0$. In view of the closed-loop differential inclusions \cref{eq:closed-loop}, we have
\begin{equation}
0 \in A(x^\star, u^\star), \quad 0 \in A(x^\star, \tilde{z}(t)), \quad t \geq 0.
\end{equation}
Recall from \cref{it:u-inject} in \cref{as:A} that the steady-state map $u^\star \to x^\star$ is injective from $K$ to $C$. It results that $\tilde{z}(t) = u^\star$ for all $t \geq 0$. Therefore, $\omega(x_0, z_0) = \{(x^\star, z^\star)\}$ and we get \cref{eq:ASb}. The last statement of the theorem follows from \cref{cor:cl_sem_gen} and the structure of the projected integral controller \cref{eq:proj-int}.
\end{proof}

We can now prove \cref{theo:strict-tracking}.
\begin{proof}[Proof of \cref{theo:strict-tracking}]
By \cref{lem:feasible} we already know that $(x^\star, u^\star)$ is an equilibrium for the closed loop \cref{eq:closed-loop}. Owing to \cref{prop:stab-detec}, it suffices to show that it is $h$-detectable in the sense of \cref{def:cl-detec}.

\emph{Step 1: Preliminaries.}
To this end, let $(x, z)$ be a strong solution to \cref{eq:closed-loop} such that $h(x, z; x^\star, u^\star) = 0$. By strict incremental impedance passivity (\cref{def:SOIIP}), $g(x, z) = g(x^\star, u^\star) = r$ {a.e}. In view of the closed-loop inclusions \cref{eq:closed-loop}, it follows that $z$ is a strong solution to the (maximal dissipative) differential inclusion $\dot{z} \in - N_K(z)$.  Then, by \cref{theo:abstract-cauchy},  $\dot{z}$ is a.e.\ the element of minimal norm in $-N_K(z)$. Since $z$ takes values in $K$, $N_K(z)$ always contains zero; hence,  $\dot{z} = 0$ a.e. 
This further implies that $z$ is equal to some constant $u^\dagger \in K$.

\emph{Step 2: Stability under constant inputs.} Let $x^\dagger \in \dom(A(\cdot, u^\dagger))$ be the corresponding steady state. By hypothesis, the steady-state pair $(x^\dagger, u^\dagger)$ is $h$-detectable; let us then show that $x(t) \to x^\dagger$. As in the proof of \cref{prop:stab-detec}, we consider the associated (nonempty) $\omega$-limit set $\omega(x(0), u^\dagger) \subset \dom(\A_r)$ and we let $(\tilde{x}, \tilde{z})$ be an arbitrary state trajectory in this set. 
It follows from the sequential characterization of $\omega$-limit sets that $\tilde{z} = u^\dagger$ for all $t\geq 0$. On the other hand, using a previous argument, we have $h(\tilde{x},\tilde{z}; x^\dagger, u^\dagger) = h(\tilde{x}, u^\dagger ; x^\dagger, u^\dagger) = 0$ a.e.  By $h$-detectability of $(x^\dagger, u^\dagger)$, this implies $\tilde{x}(t) = x^\dagger$ for all $t\geq 0$. Hence\, we have shown that $\omega(x(0),u^\dagger) = \{(x^\dagger, u^\dagger )\}$ and, thus, $x(t) \to x^\dagger$ as $t\to + \infty$. 

\emph{Step 3: Conclusion.}
As a result, $(x^\dagger, u^\dagger)$ is an equilibrium for the closed loop \cref{eq:closed-loop}: {indeed, the $\omega$-limit set  $\omega(x(0), z(0))$ is exactly $(x^\dagger, u^\dagger)$ and is invariant under $\sg{\S^r}$}. But by \cref{coro:uni-CL}, closed-loop equilibria are unique, meaning that $(x^\dagger, u^\dagger) = (x^\star, u^\star)$. In particular, $x(t) \to x^\star$ in $X$ when $t \to +\infty$, as desired.
\end{proof}

\section{Case studies}\label{sec:5}
We present {three} applications of our abstract framework to different classes of systems. The first case study, in \cref{subsec:RLC}, is to be intended as a tutorial. We consider a simple RLC circuit with a diode, described by a differential inclusion, and we want to regulate a certain voltage while controlling the input current, constrained to remain positive. 
The second case study, in \cref{subsec:pass_l}, considers a wide class of impedance passive infinite dimensional linear system. Finally, the third case study, in \cref{subsec:PDE_ql}, deals with a  parabolic partial differential equation with boundary control and employs more technical arguments from nonlinear analysis. We believe that each of these case studies contains new and nontrivial results.




\subsection{Voltage regulation for an RLC load containing a diode}\label{subsec:RLC}

\begin{figure}
\centering
\includegraphics[width = 0.8\linewidth]{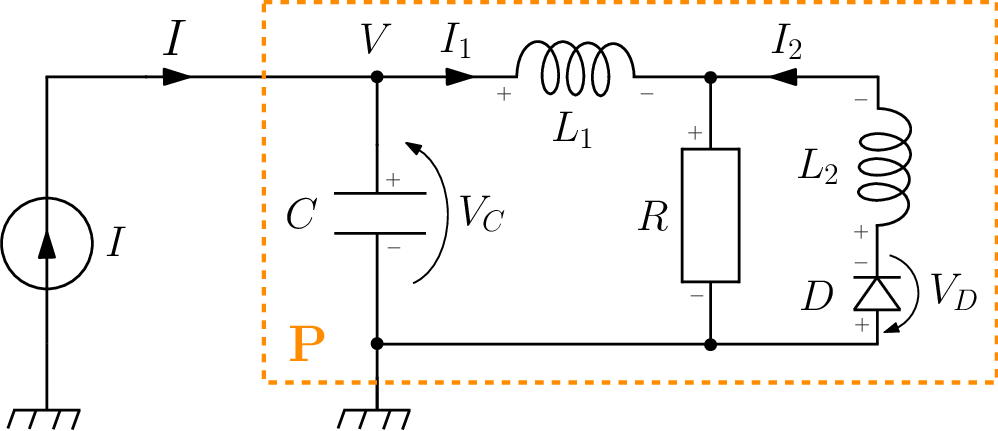}
\caption{RLC load considered in \cref{subsec:RLC}.} 
\label{fig:RLCD}
\end{figure}

We consider the output voltage regulation problem for an RLC load containing a diode, shown in \cref{fig:RLCD}. The control objective is to regulate the voltage $V$ at a desired constant reference voltage $V_{\rm ref}$, by controlling the input current $I$, which is required to be strictly positive at all times. A similar circuit is considered in \cite{Jay05}, although there are three main differences: (i) A diode $D$ is included here, which requires modeling the load through a differential inclusion, see, e.g., \cite{AcaBon10}; (ii) Capacitors and inductors are linear, since we are bound to use a quadratic energy function. (For simplicity, we consider linear resistors, although our theory would allow for nonlinear, possible multivalued, monotone resistors.) (iii) We allow for (convex) constraints on the controlled current $I$. Besides, we mention that systems similar to the RLC load considered here can be found, e.g., in \cite{AcaBon10,BroLoz20,CamHee02,CamSch16,TanBro18}.


We assume $C,L_1,L_2,R>0$ and we assume that $D$ is an ideal diode.
Under these assumptions, with the usual passive-sign convention, the load $\mathbf{P}$ from \cref{fig:RLCD} can be modelled as
\begin{align}\label{eq:P}
\mathbf{P} : \begin{cases}
&\hspace{-2mm} \dot{V}_c = -\frac{I_1}{C} +\frac{I}{C}, \\
&\hspace{-2mm} \dot{I}_1 = \frac{V_C}{L_1}-\frac{R}{L_1}(I_1+I_2), \\
&\hspace{-2mm} \dot{I}_2 = -\frac{R}{L_2}(I_1+I_2)-\frac{V_D}{L_2},\\
&\hspace{-2mm} 0\leq I_2 \perp -V_D\geq0,
\end{cases}
\end{align}
where $I_1$ and $I_2$ are the currents through the inductors $L_1$ and $L_2$, respectively, $V_D$ is the voltage across the diode, and $I$ is the input current. The last expression in \cref{eq:P} is a \textit{complementarity relationship} \cite[Section~2.3]{AcaBon10}, describing the current-voltage characteristic of the ideal diode $D$. It means: if $V_D<0$ then $I_D=I_2=0$, if $I_D=I_2>0$ then $V_D=0$.


As discussed in, e.g., \cite{BroDan06},\cite[Section~2.3.3]{AcaBon10},\cite{HeeSch00}, the complementarity relationship in \eqref{eq:P} is equivalent to
$V_D\in N_{\R^+}(I_2)$,
where $N_{\R^+}$ is the normal cone to the set $\R^+$.
By virtue of this equivalence, we define $A:\R^3\times \R\rightrightarrows \R^3$ as
\begin{align}
\label{eq:A_RLC}
A\left(V_C, I_1, I_2,I\right)&\triangleq
\bbm{-\frac{I_1}{C} +\frac{I}{C} \\
\frac{V_C}{L_1}-\frac{R}{L_1}(I_1+I_2) \\
-\frac{R}{L_2}(I_1+I_2)-\frac{1}{L_2}N_{\R^+}(I_2)},
\\ \dom(A)&=\R^2\times\R^+\times\R.
\end{align}
With this notation, $\mathbf{P}$ from \cref{eq:P} is equivalently described by 
\begin{equation}\label{eq:DI_RLC}
\quad (\dot{V}_C, \dot{I}_1, \dot{I}_2) \in A\left(V_C, I_1, I_2,I\right).
\end{equation}

The \textit{control objective} is to regulate the output voltage 
\begin{equation}\label{eq:output_RLC}
y = g\left(V_C, I_1, I_2,I\right)\triangleq V_C, \qquad \dom(g) = \R^3 \times \R,
\end{equation}
at a desired constant reference $V_{\rm ref}\in\R$ by controlling the input current $I$. We assume that the current $I$ is coming from a non-reversible power supply, hence $I$ must be kept strictly positive at all times. To solve this problem, we employ the projected integrator \cref{eq:proj-int} with $K\triangleq[I_{\min},I_{\max}]$, for some $I_{\min},I_{\max}>0$ depending on the power supply rating. 

The resulting closed-loop system is
\begin{equation}\label{eq:cl_RLC}
\begin{aligned}
(\dot{V}_C, \dot{I}_1, \dot{I}_2) &\in A\left(V_C, I_1, I_2,I\right), \\
\dot{I} &\in  - N_K(I) +V_{\rm ref} - V_C.
\end{aligned}
\end{equation}


We are now ready to state the main result of this case study.

\begin{prop}
\label{prop:RLC_conv}
Let $V_{\rm ref} \in [RI_{\min},RI_{\max}]$. Then, for all initial data $\left(V_C(0),I_1(0),I_2(0), I(0)\right) \in  \R^2\times\R^+\times K$, the corresponding solution 
to the closed loop \cref{eq:cl_RLC} satisfy
\begin{equation}
\left(V_C(t),I_1(t),I_2(t),I(t)\right) \to \left(V_{\rm ref},\frac{1}{R}V_{\rm ref},0,\frac{1}{R}V_{\rm ref}\right)
\end{equation}
as $t \to + \infty$. Moreover, $I(t)\in[I_{\min},I_{\max}]$ for all $t\geq0$.
\end{prop}

\begin{proof}
The proof is divided into steps. For convenience, we denote $x\triangleq(V_C,I_1,I_2)$, $x^\prime\triangleq(V_C^\prime,I_1^\prime,I_2^\prime)$, $x^\star\triangleq(V_C^\star,I_1^\star,I_2^\star)$.

\emph{Step 1: Structural assumptions.} We check \cref{as:A,as:compa-g,as:max-dis-c,as:dist} (summarized in \cref{table:as}).

\emph{\cref{as:A}.} Let $I^\star\in K$. We want $x^\star\in \dom(A(\cdot, I^\star))$:
\begin{equation}
0\in\bbm{-\frac{I^\star_1}{C} +\frac{I^\star}{C} \\
\frac{V_C^\star}{L_1}-\frac{R}{L_1}(I_1^\star+I_2^\star) \\
-\frac{R}{L_2}(I_1^\star+I_2^\star)-\frac{1}{L_2}N_{\R^+}(I_2^\star)}.
\end{equation}
From the first row, $I_1^\star=I^\star$. Substituting this in the third row,
\begin{equation}
-I^\star\in I_2^\star+\frac{1}{R}N_{\R^+}(I_2^\star),
\end{equation}
which has a unique solution $I_2^\star=(1+\frac{1}{R}N_{\R^+})^{-1}(-I^\star)$ since $N_{\R^+}$ is maximal monotone \cite[Theorem~12.12]{RocWet09}. It can be easily checked that $I_2^\star=0$ for all $I^\star\in K$. (Indeed, at steady-state, no current flows through $D$ in the circuit of Figure~\ref{fig:RLCD}.) Using this in the second row, we get
\begin{equation}\label{eq:equil_RLC}
x^\star=(RI^\star,I^\star,0), \quad I^\star\in K.
\end{equation}
Thus, for each $I^\star\in K$ there exists a unique equilibrium. Similarly, it follows from \cref{eq:equil_RLC}, that there necessarily exists at most one $I^\star\in K$ for each $x^\star$. This proves \cref{as:A}.

\emph{\cref{as:compa-g}.} $\dom(A)=\R^2\times\R^+\times\R\subset\R^4=\dom(g)$. 

\emph{\cref{as:max-dis-c}.} The multivalued map $\A$ from \cref{eq:cal_A} is given by $\A: \R^3 \times \R \rightrightarrows \R^3 \times \R$ defined as
\begin{align}\label{eq:cal_A_RLC}
\A\left(V_C, I_1, I_2, I\right) &\triangleq \left(A\left(V_C, I_1, I_2, I\right), -V_C\right), \\
\dom(\A) &=\left\{ \left(V_C, I_1, I_2, I\right) \in \R^2\times\R^+\times \R \right\}. \notag
\end{align}
We introduce the
inner product $\langle\cdot,\cdot\rangle_Q:\R^3\times \R^3\to\R$,
\begin{equation}\label{eq:inn_RLC}
\langle x^\prime, x\rangle_{Q}\triangleq x^TQx^\prime, \qquad Q\triangleq\operatorname{diag}(C,L_1,L_2),
\end{equation}
where $\rm{diag}$ denotes a diagonal matrix, and a decomposition of $\A$ given by
\begin{equation}\label{eq:cal_A_dec}
\A(x,I)=(A_Lx
-A_M(x)+BI, -Cx), 
\end{equation}
where 
\begin{multline}\label{eq:A_dec_mat}
A_L\triangleq\bbm{0 & -\frac{1}{C} & 0 \\ \frac{1}{L_1} & -\frac{R}{L_1} & -\frac{R}{L_1} \\ 0 & -\frac{R}{L_2} & -\frac{R}{L_2}}, \ A_M(x)\triangleq\bbm{0 \\ 0 \\ \frac{1}{L_2}N_{\R^+}(I_2)}, \\ B\triangleq\bbm{\frac{1}{C} & 0 & 0}^T, \quad C\triangleq\bbm{1 & 0 & 0}.
\end{multline}
Clearly, $-A_M$ is maximal dissipative, since $N_{\R^+}$ is maximal monotone, 
and $A_L$ is maximal dissipative, with
\begin{equation}\label{eq:R_diss}
\langle A_Lx, x\rangle_{Q}=-R(I_1+I_2)^2\leq0,
\end{equation}
which amounts to the total power dissipated in the resistor $R$ in \cref{fig:RLCD}. Moreover, $C=B^*$ with respect to the inner product  $\langle\cdot,\cdot\rangle_{Q}$. 
Combining these facts with \cite[Corollary~12.44]{RocWet09} (sum of maximal monotone operators) yields $\A$ maximal monotone. 

\emph{\cref{as:dist}.} By \cref{rem:dist}, it is sufficient to check that the steady-state input-output map $I^\star\in K\mapsto V_C^\star\in\R$ is strictly monotone. This is the case for \eqref{eq:equil_RLC} as $R>0$.

\emph{Step 2: Compactness and detectability.} \cref{as:compact} holds by \cref{rem:comp_finit_dim}. We proceed by checking detectability.

\emph{Strict output incremental passivity.} Using \cref{eq:output_RLC,eq:A_dec_mat,eq:R_diss} in the definition of $h$ from \cref{eq:h}, for any strong solutions $(x, I), (x^\prime, I^\prime)$ to \cref{eq:DI_RLC} and corresponding inputs,
\begin{equation*}\label{eq:h_RLC}
h\left(x,I;x^\prime,I^\prime\right)\triangleq R(I_1+I_2-I_1^\prime-I_2^\prime)^2+\langle f -f^\prime, I_2-I_2^\prime\rangle_\R,
\end{equation*}
where $f\in N_{\R^+}(I_2)$, $f^\prime\in N_{\R^+}(I_2^\prime$). Since $N_{\R^+}$ is monotone, 
$\langle f -f^\prime, I_2-I_2^\prime\rangle_\R\geq0$. Thus, $h\left(x,I;x^\prime,I^\prime\right)=0$ a.e.\ implies 
\begin{align}
&I_1+I_2=I_1^\prime+I_2^\prime  \quad \mbox{a.e.}, \label{eq:1} \\
&\langle f -f^\prime, I_2-I_2^\prime\rangle_\R=0  \quad \mbox{a.e.} \label{eq:2}
\end{align}
Using \cref{eq:1} in the third inclusion in \cref{eq:DI_RLC}, we get 
$$\dot{I}_2-\dot{I}_2^\prime\in-\tfrac{1}{L_2}\left(N_{\R^+}(I_2)-N_{\R^+}(I_2^\prime)\right),$$
and, using \cref{eq:2} above, we have
$$\langle \dot{I}_2 -\dot{I}_2^\prime, I_2-I_2^\prime\rangle_\R=0  \quad \mbox{a.e.}$$
Thus, either $\dot{I}_2=\dot{I}_2^\prime$ or $I_2=I_2^\prime$ a.e. Assume $\dot{I}_2=\dot{I}_2^\prime$ a.e. From \cref{eq:1}, this implies $\dot{I}_1=\dot{I}_1^\prime$ a.e., which, using the second equation in \cref{eq:DI_RLC}, together with \eqref{eq:1}, implies $V_C=V_C^\prime$ a.e. Similarly, assume $I_2=I_2^\prime$ a.e. Then, in particular, $\dot{I}_2=\dot{I}_2^\prime$ a.e. Arguing as before, $V_C=V_C^\prime$ a.e. Hence, \cref{eq:DI_RLC}--\cref{eq:output_RLC} is strictly output incrementally passive in the sense of \cref{def:SOIIP}.

\emph{All steady-state pairs are $h$-detectable.} Let $x^\star$ as in \cref{eq:equil_RLC} for some $I^\star\in K=[I_{\min},I_{\max}]$. From \cref{eq:h_RLC}, we have
\begin{equation*}
h\left(x,I^\star;x^\star,I^\star\right)= R(I_1+I_2-I_1^\star-I_2^\star)^2+\langle f -f^\star, I_2-I_2^\star\rangle_\R,
\end{equation*}
where $f\in N_{\R^+}(I_2)$, $f^\star\in N_{\R^+}(I_2^\star$). Arguing as above, using \cref{eq:equil_RLC}, $h\left(x,I^\star;x^\star,I^\star\right)=0$ a.e.\ implies 
\begin{align}
&I_1+I_2=I^\star  \quad \mbox{a.e.}, \label{eq:3} \\
&\langle f -f^\star, I_2\rangle_\R=0  \quad \mbox{a.e.} \label{eq:4}
\end{align}
Using \cref{eq:3} in the third inclusion in \cref{eq:DI_RLC}, we get 
\begin{equation}\label{eq:5}
\dot{I}_2\in-\frac{1}{L_2}\left(N_{\R^+}(I_2)+RI^\star\right) \quad \mbox{a.e.}
\end{equation}
From the above, $f\triangleq -L_2\dot{I}_2-RI^\star\in N_{\R^+}(I_2)$. Moreover, as $I_2^\star=0$, $N_{\R^+}(I_2^\star)=(-\infty,0]$. In particular, $f^\star\triangleq-RI^\star\in N_{\R^+}(I_2^\star)$. Using \cref{eq:4} with this choice of $f,f^\star$ yields
\begin{equation}\label{eq:6}
\langle \dot{I}_2,I_2\rangle_\R=0 \quad \mbox{a.e.}   
\end{equation}
Thus, either $\dot{I}_2=0$ or $I_2=0$ a.e. Assume $\dot{I}_2=0$. From \cref{eq:5}, this implies $RI^\star\in-N_{\R^+}(I_2)$. However, since $I^*>0$, this holds if and only if $I_2=0$. In turn, $I_2=0$ in \cref{eq:3} yields $I_1=I^\star$. Finally, from
the second equation in \cref{eq:DI_RLC}, $V_C= RI^\star$ a.e. Thus, $x=x^\star$. For the case $I_2=0$ it is sufficient to repeat this last line of argument. Hence, all steady-state pairs are $h$-detectable in the sense of \cref{as:OL-detec}.

\emph{Step 3: Conclusion and convergence of the output.} Clearly, any  $V_{\rm ref} \in [RI_{\min},RI_{\max}]$ is feasible in the sense of \cref{def:feasible}, since $K=[I_{\min},I_{\max}]$ and $I^\star\mapsto V_C^\star=RI^\star$. Thus, the proof concludes applying \cref{theo:strict-tracking}.
\end{proof}

\subsection{Impedance passive infinite-dimensional linear systems}\label{subsec:pass_l}
A large class of linear controlled partial differential equations can be formulated as abstract linear systems, through~\emph{system nodes}~\cite[Section~4]{TucWei14}. We show here how our results apply to this class of systems.
Let $X,Y,U$ be (real) Hilbert spaces. Consider a closed linear operator $S$ from $X \times U$ to $X \times Y$ with domain $\dom(S) \subset X \times Y$. We may then write $S$ as follows:
\begin{equation}
S \triangleq \begin{pmatrix} A\& B \\ C \& D \end{pmatrix},
\end{equation}
where $A \& B : \dom(S) \to X$, and $C\&D : \dom(S) \to Y$.


\begin{defi}[Linear system nodes]\label{def:node}
We say that the operator $S$ is a \emph{system node} on  $(U, X, Y)$ if the following hold.
\begin{enumerate}[label=(\roman*)]
    \item The operator $A : \dom(A) \subset X \to X$ defined by $Ax = A\&B(x, 0)$ for all $x \in \dom(A)$, where $\dom(A) \triangleq \{ x \in X : (x, 0) \in \dom(S) \}$, generates a $C_0$-semigroup $\{e^{tA}\}_{t\geq 0}$ of continuous linear operators on $X$.
    \item\label{it:B} Denoting by $X_{-1}$ the first \emph{extrapolation space} \cite{TucWei09book} associated with $\{e^{tA}\}_{t\geq 0}$, there exists a continuous linear operator $B : U \to X_{-1}$ such that $A\&B(x, u) = Ax + Bu$ for all $(x, u) \in \dom(S)$;
    \item $\dom(S) = \{(x, u) \in X \times U : Ax + Bu \in X \}$.
\end{enumerate}
\end{defi}
If $S$ is a system node, then $\dom(S)$ is dense in $X \times U$ and
$A\&B$ with $\dom(A\&B) = \dom(S)$ is a closed operator ~\cite[Lemma~4.7.3]{Sta05book}. Moreover, the \emph{control operator} $B$ is uniquely determined by $S$. Furthermore, we  define the \emph{observation operator}\footnote{As we will not consider convex subsets of the original state space $X$ here, we use the notation $C$ in accordance with standard linear systems terminology.}
$C$ by $Cx = C\&D(x, 0)$ for all $x \in \dom(A)$, so that $C$ is linear continuous from $\dom(A)$ into $Y$. 

In this setting,
the linear system governed by the equations
\begin{equation}
\label{eq:LTI-node}
\dot{x} = A\&B(x, u) = Ax + Bu, \quad y = C\&D(x, u).
\end{equation}
is a special case of \cref{eq:open-loop-DE}--\cref{eq:output}
if we let (with abuse in notation) $A(x, u) = A\&B(x, u) = Ax + Bu$ and $g(x, u) = C\&D(x, u)$. 
By \cite[Proposition 4.3]{TucWei14}, for any $u \in \C^2(\R^+, U)$ and initial data $x_0 \in X$ satisfying $(x_0, u(0)) \in \dom(S)$, there exists a unique \emph{classical solution} $(x, u, y)$ to \cref{eq:LTI-node} in the sense of \cite[Definition 4.2]{TucWei14}. For such solutions, $x \in \C^{1}(\R^+, X) \cap \C^{2}(\R^+, X_{-1})$, $(x(t), u(t)) \in \dom(S)$ for all $t \geq 0$ and $y \in \C(\R^+, Y)$. In particular, $x$ is  also \emph{strong solution} to \cref{eq:LTI-node} with input $u$ in the sense of \cref{def:open-loop-DE}.


\begin{defi}[Impedance passive system node]
We say that the system node $S$ is \emph{impedance passive} if
its input and output spaces coincide, i.e., $U = Y$, and if, for all classical solutions $(x, u, y)$ to \cref{eq:LTI-node},
\begin{equation}
\frac{1}{2} \dt{} \|x(t)\|^2_X \leq \langle u(t), y(t) \rangle_{Y}, \quad t > 0.
\end{equation}
\end{defi}

It is not difficult 
to show that the system node $S$ is impedance passive if and only if the associated open-loop system \cref{eq:LTI-node} is incrementally impedance passive in the sense of \cref{def:IIP}. An important characterisation of impedance passive system nodes is given by \cite[Theorem 6.2]{TucWei14}: $S$ is impedance passive \emph{if and only if} the operator
\begin{equation}
\begin{pmatrix}
A\& B \\
- C \& D
\end{pmatrix} : \dom(S) \to X \times Y
\end{equation}
is maximal dissipative. In the notation of \cref{sec:2}, this operator is precisely the map $\A$ from \cref{eq:cal_A} associated with the abstract system \cref{eq:open-loop-DE}--\cref{eq:output}. Thus, the system node $S$ is impedance passive if and only if \cref{as:max-dis-c} is satisfied. 

\begin{theo}
\label{th:si-node} Suppose that the following hold.
\begin{enumerate}[label=(\roman*)]
\item  \emph{(System node property.)} The operator $S$ is a system node on $(Y, X, Y)$.
\item \emph{(Strict output passivity.)}\label{it:SIP-node} There exists $k > 0$ such that, for all classical solutions $(x, u ,y)$ to \cref{eq:LTI-node} and for all $t>0$,
\begin{equation}
\label{eq:SysNodeStrictIP}
\frac{1}{2}\dt{}\|x(t)\|^2_X + k \|y(t)\|^2_Y \leq \langle u(t), y(t) \rangle_Y. 
\end{equation}
\item \emph{(Infinite-time approximate observability.)}\label{it:obs-node} Let $x$ 
be a classical solution to \cref{eq:LTI-node} corresponding to the zero input $u=0$. 
If $y(t) = C\&D(x(t), 0) = Cx(t) = 0$ for all $t \geq 0$ then $x(0) = 0$. 
\item \emph{(Compactness.)}\label{it:comp-node} The output space $Y$ is finite-dimensional and the semigroup generator $A$ has compact resolvent.
\item \emph{(Injectivity.)}\label{it:B-inj} The control operator $B$ has kernel $\{0\}$.
\end{enumerate}
Let $K$ be a closed convex subset of $Y$ with nonempty interior and let $r \in Y$ be a feasible reference in the sense of \cref{def:feasible}. Then, we have the following conclusions. 
\begin{enumerate}[label=(\alph*)]
\item \emph{(Closed-loop well-posedness.)} The system \cref{eq:LTI-node}
in closed loop with the projected integrator
\begin{equation}
\label{eq:PI-node}
\dot{z} \in r - y - N_K(z)
\end{equation}
gives rise to a continuous semigroup of (nonlinear) contractions on the closed convex set $X \times K$. 
\item \emph{(Asymptotic convergence.)} All generalized solutions $(x, z) \in \C(\R^+, X \times Y)$ to \cref{eq:LTI-node}--\cref{eq:PI-node} converge in $X \times Y$ to the corresponding steady-state pair $(x^\star, u^\star) \in \dom(S) \cap (X \times K)$. 
\item \emph{(Output tracking.}) For initial data $(x_0, z_0) \in \dom(S) \cap (X \times K)$, the output $y = C\&D(x, z)$ of the corresponding strong solution $(x, z)$ to \cref{eq:LTI-node}--\cref{eq:PI-node} converges to the reference $r$ in the following sense:
\begin{multline}
\label{eq:node-square-conv}
\int_T^{+\infty} \|y(t) - r\|^2_Y \, \d t
\\ \leq \frac{1}{2k}\{\|x(T) - x^\star\|^2_X + \|z(T) - u^\star\|^2_Y\} \to 0, 
\end{multline}
as $T \to + \infty$.
What is more, the map $(x_0, z_0) \mapsto y$ uniquely extends to a globally Lipschitz continuous map from $X \times K$ into $L^2(0, +\infty; Y)$, which means that \cref{eq:node-square-conv} holds  for generalized solutions as well.
\end{enumerate}
\end{theo}


\begin{proof}
The proof is divided into several steps.

\emph{Step 1: Structural assumptions.} The strict output passivity property \cref{it:SIP-node} in the theorem  implies impedance passivity of the system node $S$. By linearity and \cref{rem:SOIIP}, it is also equivalent to strict output incremental passivity as formulated in \cref{def:SOIIP}. 
We begin by proving that \cref{as:A} 
holds. To do so, we will show that $0$ lies in the resolvent set of the semigroup generator $A$, that is, $A^{-1}$ exists as a continuous linear operator on $X$. Assumption \cref{it:comp-node} of the theorem states that $A$ has compact resolvent; as a result, the spectrum of $A$ is made of \emph{eigenvalues} only. By contradiction, if $0$ is a spectral value (hence an eigenvalue), there must exist $x_0 \in \dom(A) \setminus \{0\}$ such that $Ax_0 = 0$ and, letting $x(t) \triangleq e^{tA}x_0$, we have $x(t) = x_0$ for all $t \geq 0$. But the semigroup orbit $x$ is also a classical solution to \cref{eq:LTI-node} with $u = 0$. By strict output passivity \cref{eq:PI-node}, this implies $y(t) = C\&D(x(t), 0) = Cx(t)= 0$ for all $t\geq 0$, which, in turn, forces $x(0) = x_0 = 0$ due to the observability condition \cref{it:obs-node}. Now, with $A^{-1}$ available and condition \cref{it:B-inj}, it is straightforward to check that \cref{as:A} is valid: steady states $(x^\star, u^\star) \in \dom(S)$ exist for all $u^\star \in K$ and are uniquely defined by $x^\star = - A^{-1}Bu^\star$. In particular, it makes sense to speak of feasible references $r \in Y$. In addition, using again observability and strict output passivity, we readily check that \cref{as:dist} holds. At this point, we have all the structural \cref{as:A,as:compa-g,as:max-dis-c,as:dist} (see \cref{table:as}) as well as strict output incremental passivity (\cref{def:SOIIP}). 

\emph{Step 2: Compactness and detectability.} In order to apply \cref{theo:strict-tracking}, we still need the \emph{ad-hoc} compactness properties of \cref{as:compact} along with $h$-detectability of all steady-state pairs  (\cref{as:OL-detec}). By linearity, the latter is exactly the observability property \cref{it:obs-node} in the hypotheses of \cref{th:si-node}. 
Let us check \cref{as:compact}. First, $Y$ is assumed to be finite-dimensional. Second, fix some arbitrary $\lambda > 0$. Observe that impedance passivity of the system node $S$
implies that $\{e^{tA}\}_{t \geq 0}$ is a contraction semigroup and, in particular, $\lambda$ lies in the resolvent set of $A$. Now,
let $\{(f_n, u_n)\}_{n \in \N}$ be a \emph{bounded} sequence of $X \times Y$.
The unique solution $x_n \in X$ to $\lambda x_n - A x_n = Bu_n + f_n$ is 
$x_n = (\lambda \id - A)^{-1}Bu_n + (\lambda \id - A)^{-1}f_n$. Since $Y$ is finite-dimensional, the range of $(\lambda \id - A)^{-1}B$ is finite-dimensional as well; thus, since the sequence $(\lambda \id - A)^{-1}Bu_n$ is bounded in $X$, it possesses a subsequential limit $w$ in $X$. Moreover, by hypothesis the operator $(\lambda \id - A)^{-1}$ is compact on $X$, so that $x_n$ remains bounded in $X$ and, up to a subsequence, converges in $X$ as $n \to + \infty$ to $w + v$ where $v$ is some element of $X$.

\emph{Step 3: Conclusion and convergence of the output.} Let $r \in Y$ be a feasible reference.
We can apply \cref{theo:strict-tracking} and obtain the claimed convergence result for generalized solutions to \cref{eq:LTI-node}--\cref{eq:PI-node}. Furthermore, recalling the notation $\A_r$ from \cref{eq:cal_A_r} for the generator of the (nonlinear) closed-loop semigroup $\sg{\S^r}$, we have $\dom(\A_r) = \dom(S) \cap (X \times K)$, and for initial data $(x_0, z_0) \in  \dom(\A_r)$, the corresponding strong solution $(x, z)$ to \cref{eq:LTI-node}--\cref{eq:PI-node} satisfies
\begin{equation}
\label{eq:ineq-node}
\frac{1}{2} \dt{} \{ \|x - x^\star\|^2_X + \|z - u^\star\|^2_Y \} + k \|y - r\|^2_Y \leq 0 \quad \mbox{a.e.},
\end{equation}
where $y = C\&D(x,z)$.
Since $(x, z)$ is absolutely continuous in $X \times Y$, we may integrate \cref{eq:ineq-node} over $(0, T)$, $T > 0$, and deduce \cref{eq:node-square-conv}.
With that and Cauchy sequences, we may show that the map $(x_0, z_0) \in \dom(\A_r) \mapsto y = C\&D(x, z)$, where $(x(t), z(t)) \triangleq \S_t^r(x_0, z_0)$ for all $t\geq 0$, uniquely extends by density and continuity to a globally Lipschitz continuous map from $X \times K$ into $L^2(0, +\infty; Y)$. In that regard, \cref{eq:node-square-conv} makes sense for generalized closed-loop solutions as well.
\end{proof}

\begin{rem}[Strict output passivity] Property \cref{it:SIP-node} in \cref{th:si-node} typically arises when applying negative output feedback to an impedance passive linear system; see, e.g., \cite{WeiTuc03,CurWei19}.
Indeed, if $S_0$ is an impedance passive system node on $(Y,X,Y)$ and $ k>0$, then it follows from~\cite[Theorem~4.2]{CurWei19} that $-k\id$ is a system node admissible feedback operator for $S_0$ in the sense of~\cite[Definition~7.4.2]{Sta05book}. In particular, there exists a system node $S$ on $(Y,X,Y)$ such that $(x,u) \in \dom(S)$ and $(z,y)=S (x,u)$ if and only if $(x,u-k y) \in \dom(S_0)$ and $(z,y) = S_0(x,u-ky)$. In that setting,
it is easy to see that 
$S$ is impedance passive and that the estimate \cref{eq:SysNodeStrictIP} in \cref{it:SIP-node} holds.
\end{rem}




\subsection{A nonlinear parabolic partial differential equation}\label{subsec:PDE_ql}
Let $\Omega$ be a smooth bounded domain of $\R^d$, $d \geq 2$. Let $p \geq 2$ be an even integer. Consider the following parabolic equation, inspired by \cite[page~59]{Sho13}, with boundary control:
\begin{subequations}
\label{eq:IBVP}
\begin{align}
\label{eq:PDE}
&\dl{w}{t}
- \nabla \cdot (\vnabla{w})^{p-1} + w^{p-1} = 0
&&\mbox{in}~ \Omega, \\
\label{eq:dirichlet}
& (\vec{\nabla}w)^{p-1} \cdot \vec{n} = u &&\mbox{on}~\partial \Omega,
\end{align}
\end{subequations}
where taking the power $p-1$ for vectors is meant in a component-wise sense and $\vec{n} = (n_1, \dots, n_d)$ indicates the outward unit normal vector. We assume that the boundary control $u$ has the form
\begin{equation}
\label{eq:id-u}
u = \sum_{j=1}^m u^{(j)} b_j,
\end{equation}
where $b_1, \dots, b_m$ are $m$ orthonormal profile functions in $L^2(\partial \Omega)$. We define a collocated observation by
\begin{equation}
\label{eq:heat-output}
y = \left( \int_{\partial \Omega} w b_1 \, \d \sigma, \dots,  \int_{\partial \Omega} w b_m \, \d \sigma\right) \in \R^m.
\end{equation}
Although our input and output space is $\R^m$, in order to simplify notation, we identify $\R^m$ as the closed subspace $\operatorname{span}(b_1, \dots, b_m)$ of $L^2(\partial \Omega)$ via \cref{eq:id-u}. In this spirit, we denote by $\pi_b$ the orthogonal projector on $\operatorname{span}(b_1, \dots, b_m)$.

\para{Duality maps} The variational formulation of \cref{eq:IBVP} along with \emph{a priori} energy identities suggest
working with the Hilbert space $L^2(\Omega)$, which we identify with its topological dual $L^2(\Omega)'$, and the (reflexive) Banach space $W^{1, p}(\Omega)$ in the chain of continuous and dense embeddings $W^{1,p}(\Omega) \hookrightarrow L^2(\Omega) \hookrightarrow W^{1,p}(\Omega)'$; see
\cite[Section 5.2]{Bre11book} for more details on such chains and pivot duality.
Now, for $\lambda,\mu \geq 0$, define  duality maps $\Phi_{\lambda,\mu} : W^{1,p}(\Omega)  \to W^{1,p}(\Omega)'$ as follows:
\begin{multline}
\label{eq:phi-weak}
\langle \Phi_{\lambda,\mu}(w), \varphi\rangle \triangleq  \int_\Omega (\vnabla w)^{p-1} \cdot \vnabla \varphi \, \d x + \int_\Omega w^{p - 1} \varphi \, \d x \\ 
+ \lambda \int_\Omega w \varphi \, \d x + \mu \int_{\partial \Omega} w \pi_b \varphi, \quad w, \varphi \in W^{1,p}(\Omega),
\end{multline}
where $\langle \cdot, \cdot \rangle$ denotes the duality pairing between $W^{1,p}(\Omega)$ and $W^{1,p}(\Omega)'$.
That $\Phi_{\lambda,\mu}$ is well-defined and maps into $W^{1,p}(\Omega)'$ can be verified using H\"older's inequality.
We also have 
\begin{multline}
\label{eq:incr-phi}
\langle \Phi_{\lambda,\mu}(w_1) - \Phi_{\lambda,\mu}(w_2), w_1 - w_2 \rangle  \\ = \int_\Omega ((\vnabla{w_1})^{p -1} - (\vnabla w_2)^{p-1})\cdot(\vnabla w_1 - \vnabla w_2) \, \d x  \\ + \int_\Omega(w_1^{p-1} - w_2^{p-1})(w_1 - w_2) \, \d x + \lambda \int_\Omega |w_1 - w_2|^2 \\ + \mu \int_{\partial \Omega} |\pi_b(w_1 - w_2)|^2 \, \d \sigma, \quad  w_1, w_2 \in W^{1,p}(\Omega).
\end{multline}
The next proposition summarizes key properties of $\Phi_{\lambda,\mu}$.
\begin{prop}[Duality maps]
\label{prop:dual-map}
For any $\lambda, \mu \geq 0$, we have:
\begin{enumerate}[label=(\roman*),series=phi]
        \item \label{it:phi-mono} $\Phi_{\lambda,\mu}$ is \emph{monotone}, i.e., 
$
\langle \Phi_{\lambda,\mu}(w_1) - \Phi_{\lambda,\mu}(w_2), w_1 - w_2 \rangle \geq 0$ for all $w_1, w_2 \in W^{1,p}(\Omega)$.
    \item \label{it:phi-bounded} $\Phi_{\lambda,\mu}$ is \emph{bounded}, i.e., it maps bounded sets of $W^{1,p}(\Omega)$ into bounded sets of $W^{1,p}(\Omega)'$;
    \item \label{it:phi-hemi} $\Phi_{\lambda,\mu}$ is \emph{hemicontinuous}, i.e., for fixed $w_1, w_2 \in W^{1,p}(\Omega)$,  $ 
t \in \R \mapsto\langle \Phi_{\lambda,\mu}(w_1 + t w_2), w_2\rangle$ is continuous.
\end{enumerate}
Furthermore, $\Phi_{\lambda,\mu}$ possesses a continuous inverse $\Phi_{\lambda,\mu}^{-1} : W^{1,p}(\Omega)' \to W^{1,p}(\Omega)$.
\end{prop}
\begin{rem}
In the statement of \cref{prop:dual-map}, $\lambda$ denotes the map $w \in W^{1,p}(\Omega) \mapsto \lambda w \in W^{1,p}(\Omega)'$. Recall that, in view of the embedding chain $W^{1,p}(\Omega) \hookrightarrow L^2(\Omega) \hookrightarrow W^{1,p}(\Omega)'$, $\langle w, \varphi\rangle = \int_\Omega w \varphi \, \d x$ for all $w, \varphi \in W^{1,p}(\Omega)$.
\end{rem}
\begin{proof}[Proof of \cref{prop:dual-map}] The proof is divided into steps.

\emph{Step 1: Verification of \cref{it:phi-mono,it:phi-bounded,it:phi-hemi}.}
As $p - 1$ is odd, monotonicity \cref{it:phi-mono} is an immediate consequence of \cref{eq:incr-phi}. Straightforward arguments using H\"older's inequality and trace continuity from $W^{1,p}(\Omega)$ into $L^2(\partial \Omega)$ lead to the boundedness property \cref{it:phi-bounded} and hemicontinuity \cref{it:phi-hemi}.

\emph{Step 2: Existence of the inverse.}
The space $W^{1,p}(\Omega)$ is separable. Thus, by \cite[Lemma 2.1 and Theorem 2.1]{Sho13}, properties \cref{it:phi-bounded,it:phi-hemi,it:phi-mono} imply the following: 
for any linear form $L \in W^{1,p}(\Omega)'$, if there exists $\rho > 0$ such that $\langle \Phi_{\lambda,\mu}(w), w\rangle > \langle L, w \rangle$ whenever $\|w\|_{W^{1, p}(\Omega)} > \rho$, then 
$\Phi_{\lambda,\mu}(w) = L$ possesses a solution. Here, for all $w \in W^{1,p}(\Omega)$,
\begin{multline}
\label{eq:phi-w-w}
\langle \Phi_{\lambda,\mu}(w), w\rangle =   \sum_{i =1}^{d}\int_\Omega \left | \dl{w}{x_i} \right |^{p} \, \d x + \int_\Omega |w|^p \, \d x \\
+\lambda \int_\Omega |w|^2 \, \d x + \mu \int_{\partial \Omega} |\pi_bw|^2 \, \d \sigma, 
\end{multline}
and $|\langle L, w \rangle| \leq \|L\|_{W^{1,p}(\Omega)'} \|w\|_{W^{1,p}(\Omega)}$. Since $p \geq 2$, 
\begin{equation}
\label{eq:phi-w-L}
\langle \Phi_{\lambda,\mu}(w), w\rangle - \langle L, w \rangle \to + \infty.
\end{equation}
as $\|w\|_{W^{1,p}(\Omega)} \to + \infty$.
Therefore, the equation $\Phi_{\lambda,\mu} = L$ possesses a solution for any $L \in W^{1,p}(\Omega)'$. Since $p-1$ is odd,  \cref{eq:incr-phi} yields uniqueness.
Thus,
$\Phi_{\lambda,\mu}^{-1}$ is well-defined. 

\emph{Step 3: Continuity of the inverse.} 
Let $L \in W^{1,p}(\Omega)'$ and $\{L_n\}_{n\in \N} \subset W^{1,p}(\Omega)'$ such that $L_n \to L$ in $W^{1,p}(\Omega)'$ as $n \to + \infty$. Set $w \triangleq \Phi_{\lambda,\mu}^{-1}(L)$ and $w_n \triangleq \Phi_{\lambda,\mu}^{-1}(L_n)$. 
We first prove that $w_n \to w$ in $W^{1,p}(\Omega)$ for \emph{some} subsequence and then check that the convergence has to hold for \emph{any} subsequence.
In view of \cref{eq:phi-w-w,eq:phi-w-L}, pairing $L_n$ and $w_n$ reveals that $\{w_n\}_{n \in \N}$ is bounded in $W^{1,p}(\Omega)$, which we recall is a reflexive Banach space. As as result, there exists a subsequence $\{w_{\theta(n)}\}_{n\in \N}$ and $\tilde{w} \in W^{1,p}(\Omega)$ such that
\begin{equation}
w_{\theta(n)} \rightharpoonup \tilde{w} \quad\mbox{in}~W^{1,p}(\Omega)~\mbox{weakly}, \quad n \to + \infty.
\end{equation}
Furthermore, because the elements $\Phi_{\lambda,\mu}(w_n)$ converge (strongly) in $W^{1,p}(\Omega)'$, they constitute a Cauchy sequence in $W^{1,p}(\Omega)'$; also, the $w_n - w_m$ are bounded in $W^{1,p}(\Omega)$ uniformly in $m, n \in \N$. As a result, when $m,n \to + \infty$,
\begin{equation}
\langle \Phi_{\lambda,\mu}(w_n) - \Phi_{\lambda,\mu}(w_m), w_n - w_m \rangle \to 0.
\end{equation}
In fact, $\Phi_{\lambda,\mu}$ is \emph{maximal} monotone as a map $W^{1,p}(\Omega) \rightrightarrows W^{1,p}(\Omega)'$; see \cite[Chapter 2]{Bar76book}. Therefore,
We may apply \cite[Lemma 1.3, Chapter 2]{Bar76book} and obtain $\Phi_{\lambda,\mu}(\tilde{w}) = f$ (which implies $\tilde{w} = w$) together with $w_{\theta(n)} \to w$ (strongly) in $W^{1,p}(\Omega)$ as $n \to + \infty$.
Let us finally prove that $w_n \to w$ in $W^{1,p}(\omega)$ as $n \to + \infty$, that is, the convergence holds regardless of subsequences. If not, there exist $\eps > 0$ and a subsequence $\{w_{\psi(n)}\}_{n\in \N}$ such that
\begin{equation}
\label{eq:cont-seq}
\|w_{\psi(n)} - w \|_{W^{1,p}(\Omega)} \geq \eps, \quad n \in \N.
\end{equation}
However, extracting again a subsequence,the previous arguments lead to $w_{\psi(n)} \to w$ in $W^{1,p}(\Omega)$ (strongly) as $n \to + \infty$; passing to the limit in \cref{eq:cont-seq} then gives a contradiction. 
\end{proof}

\para{Control system} In what follows, we define the map $A$ associated with \cref{eq:IBVP} as a control system with state space $L^2(\Omega)$ and input space $\R^m \simeq \operatorname{span}(b_1, \dots, b_m)$. 
Let $N : L^2(\partial \Omega) \to W^{1,p}(\Omega)'$ (as in ``Neumann'') be the linear operator
defined by, for $u \in L^2(\partial \Omega)$ and $\varphi \in W^{1,p}(\Omega)$,
\begin{equation}
\langle Nu, \varphi \rangle \triangleq \int_{\partial \Omega} u \varphi \, \d \sigma.
\end{equation}
Because the trace map $\varphi \to \varphi|_{\partial \Omega}$ is continuous from $W^{1,p}(\Omega)$ into $L^2(\partial \Omega)$, $N$ is well-defined and also bounded. In view of the variational formulation  of the boundary value problem \cref{eq:IBVP}, $Nu$ corresponds to the boundary condition $(\vnabla w)^{p-1} \cdot \vec{n} = u$ on $\partial \Omega$. Also, recall from trace theory that the space of all $\varphi|_{\partial \Omega}$ when $\varphi$ describes $W^{1,p}(\Omega)$ is exactly $W^{1-1/p,p}(\partial \Omega)$, which is dense in $L^2(\partial \Omega)$.
This shows that $N$ is injective.
Having in mind the pivot duality $W^{1,p}(\Omega) \hookrightarrow L^2(\Omega) \hookrightarrow W^{1,p}(\Omega)'$ and letting $\Phi \triangleq \Phi_{0,0}$, we can define a nonlinear map $A : L^2(\Omega) \times \R^m \to L^2(\Omega)$ as follows:
\begin{align}
A(w, u) &\triangleq - \Phi(w) + Nu, \\ 
\dom(A(\cdot, u))&=\{ w \in W^{1, p}(\Omega) : -\Phi(w) + Nu \in L^2(\Omega) \}. \notag
\end{align}
\begin{lemma}[Density]
\label{lem:comp-heat}
For all $u \in \R^{m}$, $\dom(A(\cdot, u))$ is dense in $W^{1,p}(\Omega)$. In particular, $\dom(A)$ is dense in $L^2(\Omega) \times \R^m$.
\end{lemma}
\begin{proof}
Let $u \in \R^m$, $w \in W^{1,p}(\Omega)$ and $\eps > 0$. Let $L \triangleq -\Phi(w) + Nu \in W^{1,p}(\Omega)'$. Since $L^2(\Omega)$ is dense in $W^{1,p}(\Omega)'$, there exists $L_\eps \in L^2(\Omega)$ such that $\|L_\eps - L\|_{W^{1,p}(\Omega)'} \leq \eps$. Now, let $w_\eps \triangleq \Phi^{-1}(Nu - L_\eps)$. By definition, $w_\eps \in \dom(A(\cdot, u))$. Furthermore, since $\Phi^{-1}$ is linear continuous from $W^{1,p}(\Omega)'$ into $W^{1,p}(\Omega)$, there exists $K > 0$ (independent of $w$ and $\eps$) such that
$
\|w - w_\eps\|_{W^{1,p}(\Omega)} \leq K \|L_\eps - L\|_{W^{1,p}(\Omega)'} \leq K\eps.
$
\end{proof}

On the other hand, since $\dom(A) \subset W^{1,p}(\Omega) \times \R^m$, letting $g(w, u) = g(w) = y$ as in \cref{eq:heat-output}
defines a compatible output map in accordance with \cref{as:compa-g}.
The abstract formulation of the open loop \cref{eq:IBVP}--\cref{eq:heat-output} is then
\begin{equation}
\label{eq:abstract-heat}
\dot{w} = A(w, u), \quad y = g(w).
\end{equation}
Our prior analysis of the duality maps $\Phi_{\lambda,\mu}$ makes the verification of our standing assumptions for $A$ relatively straightforward. We will in fact prove some properties that are slightly stronger than what our framework requires.
\begin{prop}[Verification of \cref{as:max-dis-c}]
The nonlinear map $\A : \dom(\A) \to L^2(\Omega) \times \to L^2(\Omega) \times \R^m$, where 
\begin{align}
\A(w, u) &\triangleq (A(w, u), -g(w)), \\ \dom(\A) &= \dom(A), \notag
\end{align}
is maximal dissipative and densely defined on $L^2(\Omega) \times \R^{m}$.
\end{prop}
\begin{proof} \Cref{lem:comp-heat} already establishes density  of $\dom(\A)$  in $L^2(\Omega) \times \R^m$ (and even $W^{1,p}(\Omega) \times \R^{m}$). Dissipativity follows from straightforward computation. 
As for the range condition, for $\lambda > 0$ and $(f, y) \in L^2(\Omega) \times \R^m$ we see that $(w, u) \in \dom(\A)$ solves $ \lambda(w, u) - \A(w, u) = (f, y)$ if and only if
\begin{subequations}
\begin{align}
\label{eq:u-lambda}
&u = \lambda^{-1}(g(w) + y), 
\\ 
\label{eq:w-lambda}
& \lambda w + \Phi(w) + \lambda^{-1}Ng(w) = \lambda^{-1}Ny.
\end{align}
\end{subequations}
We may rewrite
\cref{eq:w-lambda} as $\Phi_{\lambda,\lambda^{-1}}(w) = \lambda^{-1}Ny$, so that by \cref{prop:dual-map} there exists a unique solution $w \in W^{1,p}(\Omega)$ given by $w = \Phi_{\lambda, \lambda^{-1}}(\lambda^{-1}Ny)$. The associated $u \in \R^{m}$ is then given by \cref{eq:u-lambda} and the range condition is satisfied.
\end{proof}

\begin{prop}[Verification of \cref{as:compact}]\label{prop:w-compact}
For any $\lambda \geq 0$, the map
    \begin{equation}
    \label{eq:u-w-res}
    (f, u)   \mapsto (\lambda - A(\cdot, u))^{-1}(f) 
    \end{equation}
    is single-valued, bounded, continuous from $L^2(\Omega) \times \R^m$ into $W^{1,p}(\Omega)$, and compact from $L^2(\Omega) \times L^2(\partial \Omega)$ into $L^2(\Omega)$.
\end{prop}
\begin{proof}
Let $\lambda \geq 0$. For all $f \in L^2(\Omega)$ and $u \in \R^m$, the unique solution $w$ to $\lambda w - A(w, u) = f$ is given by
$w = \Phi^{-1}_{\lambda,0}(f + Nu)$. 
Recall that $N$ is continuous from $L^2(\partial \Omega$ into $W^{1,p}(\Omega)$ and that, by \cref{prop:dual-map}, $\Phi^{-1}_{\lambda,0}$ is continuous from $W^{1,p}(\Omega)'$ into $W^{1,p}(\Omega)'$. Continuity follows.
Furthermore, pairing $\Phi_{\lambda,0}(w)$ with $w$ and carrying out a standard line of arguments employing the Cauchy--Schwarz and Young inequalities, and also continuity from $W^{1,p}(\Omega)$ into $ L^2(\partial \Omega)$ of the trace,
we see that,
whenever $(f, u)$ describes a fixed bounded set of $L^2(\Omega) \times \R^m \simeq L^2(\Omega) \times \pi_b L^2(\partial \Omega)$,  
the variable $w$ remains bounded in $W^{1,p}(\Omega)$. Now, recall from the Rellich--Kondrachov theorem that the embedding $W^{1,p}(\Omega) \hookrightarrow L^2(\Omega)$ is compact; see, e.g., \cite[Theorem 9.16]{Bre11book}. Thus, $w$ lies in a relatively compact subset of $L^2(\Omega)$. Hence, the map 
\cref{eq:u-w-res} is compact from $L^2(\Omega) \times \R^m$ into $L^2(\Omega)$.
\end{proof}

\begin{prop}[Verification of \cref{as:A,as:dist}, unconstrained control]
\label{prop:steady-heat}
For every $u^\star \in \R^m$, there exists a unique $w^\star \in W^{1,p}(\Omega)$ such that $0 = A(w^\star, u^\star)$, given by
\begin{equation}
w^\star =  A(\cdot, u^\star)^{-1}(0) =  \Phi^{-1}(Nu^\star).
\end{equation}
Furthermore, the map $u^\star \in \R^m \mapsto w^\star \in W^{1,p}(\Omega)$ is injective. Finally, given two of such steady-state pairs, say $(w^\star, u^\star)$ and $(w^\dagger, u^\dagger)$, if 
$
\langle u^\star - u^\dagger, g(w^\star) - g(w^\dagger)\rangle_{\R^m} = 0
$
then $w^\star = w^\dagger$.
\end{prop}
\begin{proof}
The first two statements immediately follow from the proof of \cref{prop:w-compact} and the injectivity of $N$. For the final statement, we start by noting that
$
\langle u^\star - u^\dagger, g(w^\star) - g(w^\dagger)\rangle_{\R^m} = \int_{\partial \Omega} (u^\star - u^\dagger)(w^\star - w^\dagger) \, \d \sigma 
= \langle N(u^\star - u^\dagger), w^\star - w^\dagger\rangle.
$
Since $Nu^\star = \Phi(w^\star)$ and $Nu^\dagger = \Phi(w^\dagger)$, 
$\langle u^\star - u^\dagger, g(w^\star) - g(w^\dagger)\rangle_{\R^m} = \langle \Phi(w_1) - \Phi(w_2), w_1 - w_2\rangle$. From \cref{eq:incr-phi}, $\langle u^\star - u^\dagger, g(w^\star) - g(w^\dagger)\rangle_{\R^m} = 0$ yields $w^\star = w^\dagger$.
\end{proof}

\para{Strict output passivity and set-point tracking} Using \cref{eq:incr-phi}, we may show that for strong solutions $w_1,w_2$ to \cref{eq:IBVP}, with respective inputs $u_1,u_2$ (in the sense of \cref{def:open-loop-DE} for the abstract 
problem \cref{eq:abstract-heat}), the associated function $h$ is 
\begin{multline}
h(w_1, u_1; w_2, u_2) \\ = \int_\Omega ((\vnabla{w_1})^{p -1} - (\vnabla w_2)^{p-1})\cdot(\vnabla w_1 - \vnabla w_2) \, \d x \\ 
+  \int_\Omega(w_1^{p-1} - w_2^{p-1})(w_1 - w_2) \, \d x.
\end{multline}
Recalling that $p - 1$ is odd, if $h(w_1, u_1; w_2, u_2)= 0 $ a.e.\ then $w_1(t) = w_2(t)$ for all $t \geq 0$. In particular, the system \cref{eq:IBVP}--\cref{eq:heat-output} is trivially strictly output incrementally passive in the sense of \cref{def:SOIIP} and all steady-state pairs are $h$-detectable in the sense of \cref{as:OL-detec}.

\begin{theo}
Let $K$ be a closed convex subset of $\R^m$ with nonempty interior and let $r \in \R^m$ be a feasible reference. The following hold.
\begin{enumerate}[label=(\alph*)]
    \item \emph{(Closed-loop well-posedness.)}\label{it:wp-heat} The nonlinear system \cref{eq:IBVP} in closed loop with the projected integrator
    \begin{equation}
    \label{eq:PI-heat}
    \dot{z} \in r - g(w) - N_K(z),
    \end{equation}
    gives rise to a continuous semigroup of contractions on the closed convex set $L^2(\Omega) \times K$.
    \item \emph{(Asymptotic convergence.)}\label{it:conv-heat} All generalized solutions $(w, z) \in \C(\R^+, L^2(\Omega) \times \R^m)$ to \cref{eq:IBVP}--\cref{eq:PI-heat} converge in $L^2(\Omega) \times \R^m$ to the unique steady-state pair $(w^\star, u^\star) \in W^{1,p}(\Omega) \times \R^m$ satisfying $g(w^\star) = r$.
    \item \emph{(Output tracking.)} \label{it:tracking-heat} For all \emph{strong} solutions $(w, z)$ to \cref{eq:IBVP}--\cref{eq:PI-heat}, i.e., with initial data $(w_0, z_0) \in \dom(A)$, $g(w(t)) \to r$ in $\R^m$ as $t \to + \infty$.
\end{enumerate}
\end{theo}


\begin{proof}
\Cref{it:wp-heat,it:conv-heat} are immediate consequences of our abstract results and in particular \cref{theo:strict-tracking} (note that \cref{prop:steady-heat} implies that \cref{as:A,as:dist} are satisfied for \emph{any} choice of constraint set $K$). As for \cref{it:tracking-heat}, let $(w_0, z_0) \in \dom(A) = \dom(\A)$ and let $(w, z)$ be the corresponding strong solution to \cref{eq:IBVP}--\cref{eq:PI-heat}. By \cref{theo:abstract-cauchy},
\begin{multline}
\|\A_r^0(w(t), z(t))\|_{L^2(\Omega) \times \R^m} \\ 
\leq \|\A_r^0(w_0, z_0)\|_{L^2(\Omega) \times \R^m}, \quad t \geq 0,
\end{multline}
where $\A_r^0$ denotes the \emph{principal section} (see \cref{def:principal}) of the maximal dissipative closed-loop generator $\A_r$. In particular, $\|A(w, z)\|_{L^2(\Omega)}$ remains bounded. Since $z$ is bounded in $\R^m$, by \cref{prop:w-compact} $w$ must be bounded in $W^{1,p}(\Omega)$, which in turn yields boundedness in $L^2(\partial \Omega)$ of $w|_{\partial \Omega}$. In particular, $y = g(w)$ is bounded in $\R^m$. We will prove that $g(w(t)) \to r$ as $t \to + \infty$ by proving that $g(w(t_n)) \to r$ for any sequence of times $t_n \to + \infty$. Let $\{t_n\}_{n \in \N}$ be such a sequence; by compactness and since $g(w(t_n))$ is bounded in $\R^m$, it suffices to prove that \emph{any} subsequential limit of $\{g(w(t_n))\}_\N$ must be $r$. Suppose that, up to a subsequence, $g(w(t_n))$ converges to some $r' \in \R^m$ as $n \to + \infty$. Since $w(t_n)$ remains bounded in $W^{1,p}(\Omega)$ and converges (strongly) in $L^2(\Omega)$ to $w^\star$, up to another subsequence, it has to converge to $w^\star$ weakly in $W^{1,p}(\Omega)$ as well. Since the map $g$ is linear and continuous from $W^{1,p}(\Omega)$ into $\R^m$, $g(w(t_n))$ converges weakly hence strongly ($\R^m$ finite-dimensional) to $g(w^\star) = r$ as $n \to + \infty$. Uniqueness of the limit yields $r = r'$.
\end{proof}

\begin{rem}[Finite-dimensional input space and compactness]
Most of the results in this section still hold if we replace the finite-dimensional input  space $\R^m \simeq \pi_b L^2(\partial \Omega)$ with $L^2(\partial \Omega)$. What breaks down 
is compactness of the nonlinear resolvents associated with the closed loop, which we exploit in our stability analysis. However, a careful examinations of \emph{weak $\omega$-limit} sets in the spirit of \cite{Daf78} might help us overcome this issue and obtain asymptotic convergence with infinite-dimensional integrator. This is left for future work.
\end{rem}

\section{Concluding remarks}
\label{sec:conclu}
We have introduced an abstract framework for constrained set-point tracking in impedance passive infinite-dimensional nonlinear systems, focusing on the use of quadratic storage functions. There are several directions for future work.

\begin{itemize}[wide]
\item Unlike finite-dimensional approaches that allow for more general storage functions (e.g., \cite{Jay05,JayOrt07}), we find that there is a lack of tools  to handle well-posedness of infinite-dimensional dynamics that are contractive  with respect to more general metrics. We believe this to be a fundamental gap in the current literature and a promising direction for future work.
\item While the proposed framework is primarily developed with infinite-dimensional systems in mind, its applicability is relevant to finite-dimensional settings as well. For example, the case study in \cref{subsec:RLC} shows how to solve the constrained set-point tracking problem for impedance-passive circuits with multivalued dynamics. Focusing on more structured, application-driven classes of nonlinear finite-dimensional systems could allow for richer storage function choices, potentially enabling new results for nonlinear circuits coupled with differential inclusions, in the spirit of \cite{TanBro18}.
\item Finally, our \cref{as:A} rules out nontrivial steady-state behavior under constant inputs.
Nevertheless, it would be interesting to extend our framework to systems with non-unique equilibria, particularly within the setting of differential inclusions, which naturally accommodate such behaviors.
\end{itemize}

\appendices
\section{Some elements of nonlinear contraction semigroup theory}\label{ap:sg}

Let $H$ be a Hilbert space and $A : H \rightrightarrows H$ be a possibly multivalued map. Consider the abstract differential inclusion
\begin{equation}
\label{eq:abstract-DE}
\dot{x} \in A(x).
\end{equation}
\begin{defi}[Strong solutions] 
\label{def:str-sol}
A \emph{strong} solution to \cref{eq:abstract-DE} is any function $x : \R^+ \to \dom(A)$ such that $x : \R^+ \to H$ is absolutely continuous 
and \cref{eq:abstract-DE} holds a.e.\ in $(0, +\infty)$.
\end{defi}
\begin{defi}[Generalized solutions]
\label{def:gen-sol}
A \emph{generalized} solution to \cref{eq:abstract-DE} is  any function $x : \R^+ \to H$ such that, for any $T > 0$, $x$ is the limit in $\C([0, T], H)$ of a sequence of strong solutions to \cref{eq:abstract-DE}.
\end{defi}

We remark that strong solutions are also generalized solutions and generalized solutions take values in $\overline{\dom(A)}$. 

\begin{defi}[Maximal dissipative operator]
\label{def:max-dis} The map $A : H \rightrightarrows H$ is said to be \emph{maximal dissipative} if it satisfies the following properties:
\begin{enumerate}[label=(\roman*)]
    \item ({Dissipativity.}) 
    For all $x_1, x_2 \in \dom(A)$
    \begin{equation}
    \langle f_1 - f_2, x_1 - x_2\rangle_H \leq 0, 
    \end{equation}
    for all $f_1 \in A(x_1)$, $f_2 \in A(x_2)$;
    \item (Range condition.)\label{it:range-cond} There exists $\lambda > 0$ such that $\ran(\lambda \id - A) = H$.
\end{enumerate}
\end{defi}

\begin{rem}
Strictly speaking, our \cref{def:max-dis} is  a definition of \emph{$m$-dissipative} operators, which happen to coincide with maximal dissipative operators in the Hilbert space setting; see, e.g., \cite[Proposition 2.2]{Bre73book}.
\end{rem}

\begin{defi}\label{def:principal} If $A$ is maximal dissipative, its \emph{principal section} $A^0:\dom(A)\to H$ is the single-valued map 
\begin{equation}\label{eq:f}
A^0(x) \triangleq \operatorname*{arg\,min}_{f \in A(x)} \|f\|_H, \quad x \in \dom(A).
\end{equation}
\end{defi}
\Cref{def:principal} makes sense since, in the maximal dissipative case, $A(x)$ is a nonempty closed convex subset of $H$ for each $x \in \dom(A)$ \cite[Chapitre II, Section 4]{Bre73book}. We also point out that $\overline{\dom(A)}$ must be a closed convex subset of $H$ if $A$ is maximal dissipative \cite[Théorème 2.2]{Bre73book}.

The next theorem summarizes well-known results on existence and basic properties of solutions to \cref{eq:abstract-DE} when $A$ is maximal dissipative. Proofs can be found in, e.g., \cite{Bre73book,Bar76book}.
\begin{theo}
[Maximal dissipative differential inclusions]\label{theo:abstract-cauchy}
Assume that $A$ is maximal dissipative.
\begin{enumerate}[label=(\alph*)]
    \item \emph{(Strong solutions.)} Let $x_0 \in \dom(A)$. There exists a unique strong solution $x$ to \cref{eq:abstract-DE} satisfying $x(0) = x_0$. Moreover, the solution $x$ enjoys the following
    properties:
    \begin{itemize}
        \item The right derivative of $x$ (taken in $X$) exists  everywhere in $\R^+$ and
        \begin{equation}
        \label{eq:right-der}
        \frac{\d^+ x}{\d t} =A^0(x);
        \end{equation}
        \item The map $t \mapsto A^0(x(t))$ is right-continuous from $\R^+$ into $X$ and we have
        \begin{equation}
        \|A^0(x(t))\|_H \leq \|A^0(x_0)\|_H, \quad t \geq 0.
        \end{equation}
    \end{itemize}
    \item \emph{(Generalized solutions.)} Let $x_0 \in \overline{\dom(A)}$. There exists a unique generalized solution $x$ to \cref{eq:abstract-DE} with $x(0) = x_0$.
\end{enumerate}
Finally, if $x_1$ and $x_2$ are two generalized solutions to \cref{eq:abstract-DE}, 
\begin{equation}
\|x_1(t) - x_2(t)\|_H \leq \|x_1(0) - x_2(0)\|_H, \quad t \geq 0.
\end{equation}
\end{theo} 

\begin{defi}[Contraction semigroup] Let $C$ be a closed convex subset of $H$. A (continuous) \emph{contraction semigroup} is a collection $\sg{S}$ of maps $S_t : C \to C$ satisfying: 
\begin{enumerate}[label=(\roman*)]
    \item (Algebraic semigroup properties.) 
    \begin{equation}
    S_0 = \id; \quad S_{t + r} = S_t \circ S_r = S_r \circ S_t, \quad t, r \geq 0;
    \end{equation}
    \item (Continuity.) For all $x_0 \in C$, $t \mapsto S_t x_0$ is continuous from $\R^+$ into $X$.
    \item (Contraction.) For all $x_0^1, x_0^2 \in C$,
    \begin{equation}
    \|S_t x_0^1 - S_t x^2_0\|_H \leq \|x_0^1 - x_0^2\|_H, \quad t \geq 0.
    \end{equation}
\end{enumerate}
\end{defi}
\begin{coro}[Semigroup generation]
\label{coro:generation}
Suppose that $A$ is  maximal dissipative. The maps $x_0 \to x(t)$, where $t \geq 0$ and $x$ is the unique solution to \cref{eq:abstract-DE} with initial data $x_0 \in C \triangleq \overline{\dom(A)}$, define a continuous contraction semigroup $\sg{S}$ on the closed convex set $C$. Furthermore, given $x_0 \in C$,  $x_0 \in \dom(A)$ if and only if
\begin{equation}
\label{eq:generation}
\liminf_{t \to 0^+} \frac{\|S_tx_0 - x_0\|_H}{t} < + \infty,
\end{equation}
and for $x_0 \in \dom(A)$,
\begin{equation}
\label{eq:generation-bis}
\lim_{t \to 0^+} \frac{S_tx_0 - x_0}{t} = A^0x_0.
\end{equation}
\end{coro}
\begin{proof}
Existence of $\sg{S}$ readily follows from \cref{theo:abstract-cauchy}. If $x_0 \in \dom(A)$ then the limit inferior in \cref{eq:generation} is finite by \cref{eq:right-der}. Conversely, if $x_0 \in C$ and the limit inferior in \cref{eq:generation} is finite, then $x_0 \in \dom(A)$ by \cite[Théorème 3.5]{Bre73book} and \cref{eq:generation-bis} follows again from \cref{eq:right-der}.
\end{proof}
In fact, $A$ is uniquely defined by $\sg{S}$; see \cite[Théorème 4.1]{Bre73book} or \cite[Theorem 1.2, Chapter IV]{Bar76book}, which show that contraction semigroups on closed convex subset of Hilbert are always generated by maximal dissipative operators.

The next (and last) lemma characterizes equilibria of \cref{eq:abstract-DE}.
\begin{lemma}[Equilibria]
\label{lem:equilibria}
Let $A$ be maximal dissipative and let $\sg{S}$ be the associated contraction semigroup on $C \triangleq \overline{\dom(A)}$. Given $x^\star \in C$, the following are equivalent:
\begin{enumerate}[label=(\roman*)]
 \item\label{it:fixed} $S_tx^\star=x^\star$ for all $t\geq0$;
 \item \label{it:zero}
 $0\in A(x^\star)$.
\end{enumerate}
\end{lemma}

\begin{proof}
Let $x^\star \in C$ such that $S_t x^\star = x^\star$ for all $t \geq 0$. Then, in particular, $(S_t x^\star - x^\star)/t \to 0$ in $H$ as $t \to 0$, $t > 0$, i.e., $0 \in A(x^\star)$ by \cref{coro:generation}. Conversely, suppose $0 \in A(x^\star)$ and let $x(t) \triangleq x^\star$ for all $t\geq 0$. Then, $x$ takes values in $\dom(A)$, is absolutely continuous in $X$, and satisfies $(\d / \d t)x(t) = 0 \in A(x^\star) = A(x(t))$ for all $t > 0$. By uniqueness of strong solutions to \cref{eq:abstract-DE}, $x^\star = S_t x^\star$ for all $t \geq 0$.
\end{proof}

\bibliographystyle{ieeetr}
\bibliography{ProjectedIntegralControlContraction.bib}

\vfill
\end{document}